\documentclass[11pt]{article}
\usepackage[T1]{fontenc}
\usepackage{lmodern,amsmath,amsthm,amsfonts,amssymb,graphicx,float,microtype,thmtools,underscore,mathtools,anyfontsize,thm-restate,graphicx}
\usepackage[usenames,dvipsnames,svgnames,table]{xcolor}
\usepackage[unicode=true]{hyperref}
\hypersetup{
colorlinks,
linkcolor={black},
citecolor={black},
urlcolor={blue!60!black},
pdftitle={Product Structure of Graph Classes with Bounded Treewidth},
pdfauthor={Campbell--Clinch--Distel--Gollin--Hendrey--Hickingbotham--Huynh--Illingworth--Tamitegama--Tan--Wood}
}
\usepackage[noabbrev,capitalise,nameinlink]{cleveref}
\crefname{lem}{Lemma}{Lemmas}
\crefname{thm}{Theorem}{Theorems}
\crefname{cor}{Corollary}{Corollaries}
\crefname{prop}{Proposition}{Propositions}
\crefname{conj}{Conjecture}{Conjectures}
\crefname{openproblem}{Open Problem}{Open Problems}
\crefformat{equation}{(#2#1#3)}
\Crefformat{equation}{Equation #2(#1)#3}
\usepackage[shortlabels]{enumitem}
\setlist[itemize]{topsep=0ex,itemsep=0ex,parsep=0.25ex}
\setlist[enumerate]{topsep=0ex,itemsep=0ex,parsep=0.25ex}
\usepackage{subcaption}
\usepackage{standalone}
\usepackage{tikz}
\usetikzlibrary{decorations.pathreplacing}
\usetikzlibrary{arrows}
\usetikzlibrary{calc}
\usetikzlibrary{patterns}
\newcommand{\defn}[1]{\textcolor{Maroon}{\emph{#1}}}

\newcommand{\WW}{\mathcal{W}}
\newcommand{\TT}{\mathcal{T}}
\newcommand{\GG}{\mathcal{G}}
\newcommand{\KK}{\mathcal{K}}
\newcommand{\II}{\mathcal{I}}
\newcommand{\XX}{\mathcal{X}}
\newcommand{\RR}{\mathcal{R}}

\newcommand{\HH}{\mathcal{H}}
\newcommand{\FF}{\mathcal{F}}
\newcommand{\JJ}{\mathcal{J}}

\newcommand{\PP}{\mathcal{P}}

\newcommand{\NN}{\mathbb{N}}
\newcommand{\REALS}{\mathbb{R}}
\newcommand{\JournalArxiv}[2]{#1}

\usepackage[longnamesfirst,numbers,sort&compress]{natbib}
\makeatletter
\def\NAT@spacechar{~}
\makeatother
\usepackage[bmargin=25mm,tmargin=25mm,lmargin=30mm,rmargin=30mm]{geometry}
\allowdisplaybreaks

\DeclarePairedDelimiter{\abs}{\lvert}{\rvert}
\DeclarePairedDelimiter{\floor}{\lfloor}{\rfloor}

\renewcommand{\geq}{\geqslant}
\renewcommand{\leq}{\leqslant}
\renewcommand{\emptyset}{\varnothing}
\DeclareMathOperator{\dist}{dist}
\DeclareMathOperator{\diam}{diam}

\DeclareMathOperator{\tw}{tw}
\DeclareMathOperator{\pw}{pw}

\DeclareMathOperator{\bw}{bw}
\DeclareMathOperator{\tpw}{tpw}
\DeclareMathOperator{\td}{td}

\renewcommand{\thefootnote}{\fnsymbol{footnote}}
\allowdisplaybreaks
\theoremstyle{plain}
\newtheorem{thm}{Theorem}
\newtheorem{lem}[thm]{Lemma}
\newtheorem*{claim}{Claim}
\newtheorem{cor}[thm]{Corollary}
\newtheorem{prop}[thm]{Proposition}
\newtheorem{obs}[thm]{Observation}
\theoremstyle{definition}

\newenvironment{proofofclaim}[1][Proof.]{%
    \begin{proof}[{#1}]%
        }{%
    \end{proof}}

\begin{document}

\author{
Rutger~Campbell\,\footnotemark[3] 
\qquad Katie~Clinch\,\footnotemark[6]
\qquad Marc~Distel\,\footnotemark[2]
\qquad J.~Pascal Gollin\,\footnotemark[3]\\
Kevin~Hendrey\,\footnotemark[3]
\qquad Robert~Hickingbotham\,\footnotemark[2]
\qquad Tony Huynh\,\footnotemark[2] \\
Freddie Illingworth\,\footnotemark[5] 
\qquad Youri Tamitegama\,\footnotemark[5]
\qquad Jane Tan\,\footnotemark[5]
\qquad David~R.~Wood\,\footnotemark[2]}

\footnotetext[3]{Discrete Mathematics Group, Institute for Basic Science (IBS), Daejeon, Republic~of~Korea (\texttt{\{rutger, kevinhendrey, pascalgollin\}@ibs.re.kr}). Research of R.C.\ and K.H.\ supported by the Institute for Basic Science (IBS-R029-C1). Research of J.P.G.\ supported by the Institute for Basic Science (IBS-R029-Y3).}
\footnotetext[6]{Department of Electrical and Electronic Engineering, University of Melbourne, Australia (\texttt{katie.clinch\allowbreak@unimelb.edu.au}).}
\footnotetext[2]{School of Mathematics, Monash University, Melbourne, Australia  (\texttt{\{robert.hickingbotham, tony.huynh2, david.wood\}@monash.edu, mdistel24@gmail.com}). Research of D.W.\ supported by the Australian Research Council and a Visiting Research  Fellowship of Merton College, University of Oxford. Research of R.H.\ supported by an Australian Government Research Training Program Scholarship.}
\footnotetext[5]{Mathematical Institute, University of Oxford, United Kingdom (\texttt{\{illingworth, tamitegama, jane.tan\}\allowbreak@maths.ox.ac.uk}). Research of F.I.\ supported by EPSRC grant EP/V007327/1.}

\sloppy
	
\title{\textbf{Product Structure of Graph Classes\\ with Bounded Treewidth}}
	
\maketitle

\begin{abstract}
    We show that many graphs with bounded treewidth can be described as subgraphs of the strong product of a graph with smaller treewidth and a bounded-size complete graph. 
    To this end, define the \defn{underlying treewidth} of a graph class~$\GG$ to be the minimum non-negative integer~$c$ such that, for some function~$f$, for every graph~${G \in \GG}$ there is a graph~$H$ with~${\tw(H) \leq c}$ such that~$G$ is isomorphic to a subgraph of~${H \boxtimes K_{f(\tw(G))}}$. 
    We introduce disjointed coverings of graphs and show they determine the underlying treewidth of any graph class. 
    Using this result, we prove that the class of planar graphs has underlying treewidth~$3$; 
    the class of $K_{s,t}$-minor-free graphs has underlying treewidth~$s$ (for~${t \geq \max\{s,3\}}$); 
    and the class of $K_t$-minor-free graphs has underlying treewidth~${t-2}$. 
    In general, we prove that a monotone class has bounded underlying treewidth if and only if it excludes some fixed topological minor. 
    We also study the underlying treewidth of graph classes defined by an excluded subgraph or excluded induced subgraph. 
    We show that the class of graphs with no~$H$ subgraph has bounded underlying treewidth if and only if every component of~$H$ is a subdivided star, and that the class of graphs with no induced~$H$ subgraph has bounded underlying treewidth if and only if every component of~$H$ is a star.
\end{abstract}


\renewcommand{\thefootnote}{\arabic{footnote}}

\newpage
\section{Introduction}
\label{Introduction}

Graph product structure theory describes complicated graphs as subgraphs of strong products\footnote{The \defn{strong product} of graphs~$A$ and~$B$, denoted by~${A \boxtimes B}$, is the graph with vertex-set~${V(A) \times V(B)}$, where distinct vertices ${(v,x),(w,y) \in V(A) \times V(B)}$ are adjacent if
	${v=w}$ and ${xy \in E(B)}$, or
	${x=y}$ and ${vw \in E(A)}$, or
	${vw \in E(A)}$ and~${xy \in E(B)}$.
} of simpler building blocks. 
The building blocks typically have bounded treewidth, which is the standard measure of how similar a graph is to a tree. 
Examples of graphs classes that can be described this way include planar graphs \citep{DJMMUW20,UWY22}, graphs of bounded Euler genus~\citep{DJMMUW20,DHHW}, graphs excluding a fixed minor~\citep{DJMMUW20}, and various non-minor-closed classes~\citep{DMW,HW21b}. 
These results have been the key to solving several open problems regarding queue layouts~\citep{DJMMUW20}, nonrepetitive colouring~\citep{DEJWW20}, $p$-centered colouring~\citep{DFMS21}, adjacency labelling~\citep{EJM,DEJGMM21},  twin-width~\citep{BKW,BDHK}, and 
comparable box dimension~\citep{DGLTU22}. 

This paper shows that graph product structure theory can even be used to describe graphs of bounded treewidth in terms of simpler graphs. 
Here the building blocks are graphs of smaller treewidth and complete graphs of bounded size. 
For example, a classical theorem by the referee of \citep{DO95} can be interpreted as saying that every graph~$G$ of treewidth~$k$ and maximum degree~$\Delta$ is contained\footnote{A graph~$G$ is \defn{contained} in a graph~$X$ if~$G$ is isomorphic to a subgraph of~$X$.} in ${T \boxtimes K_{O(k\Delta)}}$ for some tree~$T$. 

This result motivates the following definition. 
The \defn{underlying treewidth} of a graph class~$\GG$ is the minimum~${c \in \NN_0}$ such that, for some function~$f$, for every graph~${G \in \GG}$ there is a graph~$H$ with~${\tw(H) \leq c}$ such that~$G$ is contained in~${H \boxtimes K_{f(\tw(G))}}$. 
If there is no such~$c$, then~$\GG$ has \defn{unbounded} underlying treewidth. 
We call~$f$ the \defn{treewidth-binding function}. 
For example, the above-mentioned result in \citep{DO95} says that any graph class with bounded degree has underlying treewidth at most~1 with treewidth-binding function~${O(k)}$. 

This paper introduces disjointed coverings of graphs and shows that they are intimately related to underlying treewidth; see \cref{DisjointedCoverings}. 
Indeed, we show that disjointed coverings characterise the underlying treewidth of any graph class (\cref{UnderlyingTreewidthDisjointedPartition}). 
The remainder of the paper uses disjointed coverings to determine the underlying treewidth of several graph classes of interest, with a small treewidth-binding function as a secondary goal. 

\textbf{Minor-closed classes:} 
We prove that every planar graph of treewidth~$k$ is contained in~${H \boxtimes K_{O(k^2)}}$ where~$H$ is a graph of treewidth~3. Moreover, this bound on the treewidth of~$H$ is best possible. 
Thus the class of planar graphs has underlying treewidth~$3$ (\cref{planarunderlying}). 
We prove the following generalisations of this result: 
the class of graphs embeddable on any fixed surface has underlying treewidth~$3$ (\cref{SurfaceProduct}); 
the class of $K_{t}$-minor-free graphs has underlying treewidth~${t-2}$ (\cref{NoKtMinorUnderlyingTreewidth}); 
and for~${t \geq \max\{s,3\}}$ the class of $K_{s,t}$-minor-free graphs has underlying treewidth~$s$ (\cref{NoKstMinorUnderlyingTreewidth}). 
In all these results, the treewidth-binding function is~${O(k^2)}$ for fixed~$s$ and~$t$. 

\textbf{Monotone Classes:} 
We characterise the monotone graph classes with bounded underlying treewidth. 
We show that a monotone graph class~$\GG$ has bounded underlying treewidth if and only if~$\GG$ excludes some fixed topological minor (\cref{MonotoneUnderlyingTreewidth}). 
In particular, we show that for~${t \geq 5}$ the class of $K_t$-topological-minor-free graphs has underlying treewidth~$t$ (\cref{TopoMinorUnderlyingTreewidth}). 
The characterisation for monotone classes has immediate consequences. For example, it implies that the class of 1-planar graphs has unbounded underlying treewidth. 
On the other hand, for any~${k \in \NN}$, we show that the class of outer $k$-planar graphs has underlying treewidth~2 (\cref{OuterkPlanarthm}), which generalises the well-known fact that outerplanar graphs have treewidth~2. 

We use our result for disjointed coverings to characterise the graphs~$H$ for which the class of $H$-free graphs has bounded underlying treewidth. 
In particular, the class of $H$-free graphs has bounded underlying treewidth if and only if every component of~$H$ is a subdivided star (\cref{SummaryTheorem}). 
For specific graphs~$H$, including paths and disjoint unions of paths, we precisely determine the underlying treewidth of the class of $H$-free graphs. 

\textbf{Hereditary Classes:} 
We characterise the graphs~$H$ for which the class of graphs with no induced subgraph isomorphic to~$H$ has bounded underlying treewidth. 
The answer is precisely when every component of~$H$ is a star, in which case the underlying treewidth is at most~2. 
Moreover, we characterise the graphs~$H$ for which the class of graphs with no induced subgraph isomorphic to~$H$ has underlying treewidth~0, 1 or~2 (\cref{InducedSubgraphUnderlyingTreewidth}). 

\textbf{Universal Graphs:} 
A graph~$U$ is \defn{universal} for a graph class~$\GG$ if~${U \in \GG}$ and~$U$ contains every graph in~$\GG$. 
This definition is only interesting when considering infinite graphs. 
For each~${k \in \NN}$ there is a universal graph~$\TT_k$ for the class of countable graphs of treewidth~$k$. 
\citet{HMSTW} gave an explicit construction for~$\TT_k$, and showed how product structure theorems for finite graphs lead to universal graphs. 
Their results imply that for any hereditary class~$\GG$ of countable graphs, if the class of finite graphs in~$\GG$ has underlying treewidth~$c$ with treewidth-binding function~$f$, then every graph in~$\GG$ of treewidth at most~$k$ is contained in~${\TT_c \boxtimes K_{f(k)}}$. 
This result is applicable to all the classes above. 
For example, every countable $K_t$-minor free graph of treewidth~$k$ is contained in~${\TT_{t-2} \boxtimes K_{O(k^2)}}$.

\JournalArxiv{We omit straightforward proofs from the main body of the paper; full details can be found in \cref{Omitted}, which is not intended for publication in the journal.}{\textbf{Other Underlying Parameters:} The definition of underlying treewidth suggests an underlying version of any graph parameter. \cref{UnderlyingParameters} explores this idea, focusing on underlying chromatic number.}

\section{Preliminaries}

\subsection{Basic Definitions}
\label{Basics}

See \citep{Diestel5} for graph-theoretic definitions not given here. 
We consider simple, finite, undirected graphs~$G$ with vertex-set~${V(G)}$ and edge-set~${E(G)}$. 
A \defn{graph class} is a collection of graphs closed under isomorphism. 
A graph class is \defn{hereditary} if it is closed under taking induced subgraphs. 
A graph class is \defn{monotone} if it is closed under taking subgraphs. 
A graph~$H$ is a \defn{minor} of a graph~$G$ if~$H$ is isomorphic to a graph obtained from a subgraph of~$G$ by contracting edges. 
A graph~$G$ is \defn{$H$-minor-free} if~$H$ is not a minor of~$G$. 
A graph class~$\GG$ is \defn{minor-closed} if every minor of each graph in~$\GG$ is also in~$\GG$. 

The class of planar graphs is minor-closed. 
More generally, the class of graphs embeddable on a given surface (that is, a closed  compact 2-manifold) is minor-closed. 
The \defn{Euler genus} of a surface with~$h$ handles and~$c$ cross-caps is~${2h+c}$. 
The \defn{Euler genus} of a graph~$G$ is the minimum~${g \in \NN_0}$ such that there is an embedding of~$G$ in a surface of Euler genus~$g$; 
see \cite{MoharThom} for more about graph embeddings in surfaces. 
A graph is \defn{linklessly embeddable} if it has an embedding in~$\REALS^3$ with no two linked cycles; 
see \citep{RST93a} for a survey and precise definitions. 
The class of linklessly embeddable graphs is also minor-closed. 

A graph~$\tilde{G}$ is a \defn{subdivision} of a graph~$G$ if~$\tilde{G}$ can be obtained from~$G$ by replacing each edge~${vw}$ by a path~$P_{vw}$ with endpoints~$v$ and~$w$ (internally disjoint from the rest of~$\tilde{G}$). 
If each~$P_{vw}$ has~$t$ internal vertices, then~$\tilde{G}$ is the \defn{$t$-subdivision} of~$G$. 
If each~$P_{vw}$ has at most~$t$ internal vertices, then~$\tilde{G}$ is a \defn{$(\leq t)$-subdivision} of~$G$. 
A graph~$H$ is a \defn{topological minor} of~$G$ if a subgraph of~$G$ is isomorphic to a subdivision of~$H$. 
A graph~$G$ is \defn{$H$-topological-minor-free} if~$H$ is not a topological minor of~$G$. 

A \defn{clique} in a graph is a set of pairwise adjacent vertices. 
Let~\defn{$\omega(G)$} be the size of the largest clique in a graph~$G$. 
An \defn{independent set} in a graph is a set of pairwise non-adjacent vertices. 
Let~\defn{$\alpha(G)$} be the size of the largest independent set in a graph~$G$. 
Let~\defn{$\chi(G)$} be the chromatic number of~$G$. 
Note that~${\abs{V(G)} \leq \chi(G)\alpha(G)}$. 
A graph~$G$ is \defn{$d$-degenerate} if every non-empty subgraph of~$G$ has a vertex of degree at most~$d$. 
A greedy algorithm shows that~${\chi(G) \leq d+1}$ for every $d$-degenerate graph~$G$.

Let~\defn{$P_n$} be the $n$-vertex path. 
For a graph~$G$ and~${\ell \in \NN}$, let~\defn{$\ell\,G$} be the union of~$\ell$ vertex-disjoint copies of~$G$. 
Let \defn{$\widehat{G}$} be the graph obtained from~$G$ by adding one dominant vertex. 

Let~${\NN \coloneqq \{1,2,\dots\}}$ and~${\NN_0 \coloneqq \{0,1,\dots\}}$. 
All logarithms in this paper are binary.

\subsection{Tree-Decompositions}

For a tree~$T$, a \defn{$T$-decomposition} of a graph~$G$ is a collection~${\WW = (W_x \colon x \in V(T))}$ of subsets of~${V(G)}$ indexed by the nodes of~$T$ such that
(i) for every edge~${vw \in E(G)}$, there exists a node~${x \in V(T)}$ with~${v,w \in W_x}$; and 
(ii) for every vertex~${v \in V(G)}$, the set~${\{ x \in V(T) \colon v \in W_x \}}$ induces a (connected) subtree of~$T$. 
Each set~$W_x$ in~$\WW$ is called a \defn{bag}. 
The \defn{width} of~$\WW$ is~${\max\{ \abs{W_x} \colon x \in V(T) \}-1}$. 
A \defn{tree-decomposition} is a $T$-decomposition for any tree~$T$. 
The \defn{treewidth}~$\tw(G)$ of a graph~$G$ is the minimum width of a tree-decomposition of~$G$. 
Treewidth is the standard measure of how similar a graph is to a tree. 
Indeed, a connected graph has treewidth 1 if and only if it is a tree. 
Treewidth is of fundamental importance in structural and algorithmic graph theory; see \citep{Reed03,HW17,Bodlaender-TCS98} for surveys. 

We use the following well-known facts about treewidth. 
Every minor~$H$ of a graph~$G$ satisfies~${\tw(H) \leq \tw(G)}$. 
In every tree-decomposition of a graph~$G$, each clique of~$G$ appears in some bag. 
Thus~${\tw(G) \geq \omega(G)-1}$
and~${\tw(K_n) = n-1}$. 
If~${\{v_1,\dots,v_k\}}$ is a clique in a graph~$G_1$ and~${\{w_1,\dots,w_k\}}$ is a clique in a graph~$G_2$, and~$G$ is the graph obtained from the disjoint union of~$G_1$ and~$G_2$ by identifying~$v_i$ and~$w_i$ for each~${i \in \{1,\dots,k\}}$, then~${\tw(G) = \max\{\tw(G_1),\tw(G_2)\}}$. 
For any graph~$G$, we have~${\tw(\widehat{G}) = \tw(G) + 1}$ and ${\tw(\ell\, G) = \tw(G)}$ for any~${\ell \in \NN}$, implying~${\tw(\widehat{\ell\, G}) = \tw(G)+1}$. 
Finally, every graph~$G$ is $\tw(G)$-degenerate, implying~${\chi(G) \leq \tw(G)+1}$.

\subsection{Partitions}
\label{Partitions}

To describe our main results in \cref{Introduction}, it is convenient to use the language of graph products. 
However, to prove our results, it is convenient to work with the equivalent notion of graph partitions, which we now introduce. 

For graphs~$G$ and~$H$, an \defn{$H$-partition} of~$G$ is a partition~${(V_x \colon x\in V(H))}$ of~${V(G)}$ indexed by the nodes of~$H$, such that for every edge~${vw}$ of~$G$, if~${v \in V_x}$ and~${w \in V_y}$, then~${x = y}$ or~${xy \in E(H)}$. 
We say that~$H$ is the \defn{quotient} of such a partition. 
The \defn{width} of such an $H$-partition is~${\max\{ \abs{V_x} \colon x \in V(H)\}}$. 
For~${c \in \NN_0}$, an $H$-partition where~${\tw(H) \leq c}$ is called a \defn{$c$-tree-partition}. 
The \defn{$c$-tree-partition-width} of a graph~$G$, denoted \defn{$\tpw_c(G)$}, is the minimum width of a $c$-tree-partition of~$G$. 

It follows from the definitions that a graph~$G$ has an $H$-partition of width at most~$\ell$ if and only if~$G$ is contained in~${H \boxtimes K_\ell}$. 
Thus, $\tpw_c(G)$ equals the minimum~${\ell \in \NN_0}$ such that~$G$ is contained in~${H \boxtimes K_{\ell}}$ for some graph~$H$ with~${\tw(H) \leq c}$. 
Hence, the underlying treewidth of a graph class~$\GG$ equals the minimum~${c \in \NN_0}$ such that, for some function~$f$, every graph~${G \in \GG}$ has $c$-tree-partition-width at most~${f(\tw(G))}$. 
We henceforth use this as our working definition of underlying treewidth. 

If a graph~$G$ has an $H$-partition for some graph~$H$ of treewidth~$c$, then we may assume that~$H$ is edge-maximal of treewidth~$c$. 
So~$H$ is a $c$-tree (which justifies the `$c$-tree-partition' terminology). 
Such graphs~$H$ are chordal. 
Chordal partitions are well studied with several applications~\citep{vdHW18,SSW19,HOQRS17,RS98,HMSTW}. 
For example, \citet{vdHW18} proved that every $K_t$-minor-free graph has a ${(t-2)}$-tree-partition in which each part induces a connected subgraph with maximum degree at most~${t-2}$ (amongst other properties). 
Our results give chordal partitions with bounded-size parts (for graphs of bounded treewidth). 

Before continuing, we review work on the~${c = 1}$ case.
A \defn{tree-partition} is a $T$-partition for some tree~$T$. 
The \defn{tree-partition-width}\footnote{Tree-partition-width has also been called \defn{strong treewidth} \citep{BodEng-JAlg97, Seese85}.} of~$G$, denoted by~${\tpw(G)}$, is the minimum width of a tree-partition of~$G$. 
Thus~${\tpw(G) = \tpw_1(G)}$, which equals the minimum~${\ell \in \NN_0}$ for which~$G$ is contained in~${T \boxtimes K_\ell}$ for some tree~$T$. 
Tree-partitions were independently introduced by \citet{Seese85} and \citet{Halin91}, and have since been widely investigated \citep{Bodlaender-DMTCS99,BodEng-JAlg97, DO95,DO96,Edenbrandt86, Wood06,Wood09}. 
Applications of tree-partitions include 
graph drawing~\citep{CDMW08,GLM05,DMW05,DSW07,WT07}, 
nonrepetitive graph colouring~\citep{BW08}, 
clustered graph colouring~\citep{ADOV03}, 
monadic second-order logic~\citep{KuskeLohrey05}, 
network emulations~\citep{Bodlaender-IPL88, Bodlaender-IC90, BvL-IC86, FF82}, statistical learning theory~\citep{ZA22}, and the edge-{E}rd{\H{o}}s-{P}{\'o}sa property~\citep {RT17,GKRT16,CRST18}. 
Planar-partitions and other more general structures have also been studied~\citep{DK05,RS98,WT07,DOSV00,DOSV98}.

Bounded tree-partition-width implies bounded treewidth, as noted by \citet{Seese85}. This fact easily generalises for $c$-tree-partition-width; 
see \cref{Omitted} for a proof. 

\begin{restatable}{obs}{cTreePartitionWidthTreewidth}
    \label{cTreePartitionWidthTreewidth}
    For every graph~$G$ and~${c\in\NN_0}$, we have~${\tw(G) \leq (c+1)\tpw_c(G)-1}$.
\end{restatable}

Of course, ${\tw(T) = \tpw(T) = 1}$ for every tree~$T$. 
But in general, $\tpw(G)$ can be much larger than~$\tw(G)$. 
For example, fan graphs on~$n$ vertices have treewidth~2 and tree-partition-width~$\Omega(\sqrt{n})$; 
see \cref{StandardExamples} below. 
On the other hand, the referee of \citep{DO95} showed that if the maximum degree and treewidth are both bounded, then so is the tree-partition-width, which is one of the most useful results about tree-partitions. 
 
\begin{lem}[\citep{DO95}]
    \label{TreePartitionWidthDegree}
    For~${k,\Delta\in\NN}$, every graph of treewidth less than~$k$ and maximum degree at most~$\Delta$ has tree-partition-width at most~${24k\Delta}$.  
\end{lem}
 
This bound is best possible up to the multiplicative constant~\citep{Wood09}. 
Note that bounded maximum degree is not necessary for bounded tree-partition-width (for example, stars). 
\citet{DO96} characterised graph classes with bounded tree-partition-width in terms of excluded topological minors. 
We give an alternative characterisation, which says that graph classes with bounded tree-partition-width are exactly those that have bounded treewidth and satisfy a further `disjointedness' condition. 
Furthermore, this result naturally generalises for $c$-tree-partition-width and thus for underlying treewidth.

\section{Disjointed Coverings}
\label{DisjointedCoverings}

This section introduces disjointed coverings and shows that they can be used to characterise bounded $c$-tree-partition-width and underlying treewidth. 

An \defn{$\ell$-covering} of a graph~$G$ is a set~${\beta \subseteq 2^{V(G)}}$ such that~${\abs{B} \leq \ell}$ for every~${B \in \beta}$, 
and~${\cup\{B \colon B\in\beta\} = V(G)}$.\footnote{Our definition of $\ell$-covering differs from the standard usage where it refers to a covering in which each element of the ground set is covered $\ell$ times.} 
If~${B_1 \cap B_2 = \emptyset}$ for all distinct~${B_1,B_2 \in \beta}$, then~$\beta$ is an \defn{$\ell$-partition}. 
As illustrated in \cref{DisjointedCovering1}, an $\ell$-covering~$\beta$ of a graph~$G$ is \defn{$(c,d)$-disjointed} if for every \mbox{$c$-tuple} ${(B_1,\dots,B_c) \in \beta^c}$ and every component~$X$ of~${G-(B_1\cup\dots\cup B_c)}$ there exists~${Q \subseteq V(X)}$ with~${\abs{Q} \leq d}$ such that for each component~$Y$ of~${X-Q}$, for some~${i \in\{1,\dots,c\}}$ we have~${V(Y)\cap N_G(B'_i) = \emptyset}$, where~${B_i' \coloneqq B_i\setminus(B_1 \cup \dotsb \cup B_{i-1})}$. 
Note that we can take~${Q = \emptyset}$ if some~${B_i' = \emptyset}$, since~${N_G(\emptyset) = \emptyset}$. 

\begin{figure}[ht]
    \centering
    \includegraphics{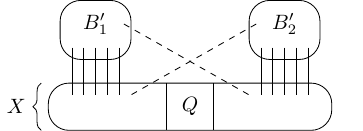}
\caption{A disjointed partition with $c=2$, where non-edges are dashed.}
    \label{DisjointedCovering1}
\end{figure}

Let~$\beta$ be an $\ell$-covering of a graph~$G$. 
For~${t \in \NN}$, let~\defn{${\beta[t]}$}~${\coloneqq \{\bigcup\mathcal{B} \colon \mathcal{B}\subseteq \beta, \abs{\mathcal{B}}\leq t\}}$. So~${\beta[t]}$ is a ${t\ell}$-covering of~$G$. 
For a function~${f \colon \NN \to \REALS^+}$ we say that~$\beta$ is \defn{$(c,f)$-disjointed} if~${\beta[t]}$ is ${(c,f(t))}$-disjointed for every~${t \in \NN}$. 

While ${(c,d)}$-disjointed coverings are conceptually simpler than ${(c,f)}$-disjointed coverings, we show  they are roughly equivalent (\cref{cTreePartitionWidthDisjointedCovering}). Moreover, ${(c,f)}$-disjointed coverings are essential for the main proof (\cref{Main}) and give better bounds on the $c$-tree-partition-width, leading to smaller treewidth-binding functions when determining the underlying treewidth of several graph classes of interest (for $K_t$-minor-free graphs for example).

Note that we often consider the singleton partition~${\beta \coloneqq \{\{v\} \colon v \in V(G)\}}$ of a graph~$G$, which is ${(c,f)}$-disjointed if and only if, for every~${t \in \NN}$, every $t$-partition of~$G$ is ${(c,f(t))}$-disjointed. 

This section characterises $c$-tree-partition-width in terms of ${(c,d)}$-disjointed coverings (or partitions) and ${(c,f)}$-disjointed coverings (or partitions). The following observation deals with the~${c = 0}$ case.

\begin{obs}
    \label{c0}
    The following are equivalent for any graph~$G$ and~${d \in \NN}$:
    \begin{itemize}
        \item $G$ has a ${(0,d)}$-disjointed covering\textnormal{;}
        \item every covering of~$G$ is ${(0,d)}$-disjointed\textnormal{;}
        \item each component of~$G$ has at most~$d$ vertices\textnormal{;} 
        \item $G$ has 0-tree-partition-width at most~$d$.
    \end{itemize}
\end{obs}

\cref{c0} implies that a graph class~$\GG$ has underlying treewidth~0 if and only if there is a function~$f$ such that every component of every graph~${G \in \GG}$ has at most~${f(\tw(G))}$ vertices. 

We prove the following characterisation of bounded $c$-tree-partition-width (which is new even in the~${c = 1}$ case). 

\begin{thm}
    \label{cTreePartitionWidthDisjointedCovering} 
    For fixed~${c \in \NN_0}$, the following are equivalent for a graph class~$\GG$ with bounded treewidth: 
    \begin{enumerate}[label=\upshape{(\alph*)}]
        \item \label{item:cRPWDC-cTPW} $\GG$ has bounded $c$-tree-partition-width\textnormal{;}
        \item \label{item:cRPWDC-cdDP} for some~${d,\ell \in \NN}$, every graph in~$\GG$ has a ${(c,d)}$-disjointed $\ell$-partition\textnormal{;}
        \item \label{item:cRPWDC-cdDC} for some~${d,\ell \in \NN}$, every graph in~$\GG$ has a ${(c,d)}$-disjointed $\ell$-covering\textnormal{;}
        \item \label{item:cRPWDC-cfDP} for some~${\ell \in \NN}$ and function~$f$, every graph in~$\GG$ has a ${(c,f)}$-disjointed $\ell$-partition\textnormal{;}
        \item \label{item:cRPWDC-cfDC} for some~${\ell \in \NN}$ and function~$f$, every graph in~$\GG$ has a ${(c,f)}$-disjointed $\ell$-covering.
    \end{enumerate}
\end{thm}

\begin{proof}
    \cref{c0} handles the~${c = 0}$ case. 
    Now assume that~${c \geq 1}$. 
    \cref{disjointed} below says that \ref{item:cRPWDC-cTPW} implies~\ref{item:cRPWDC-cdDP}. 
    Since every $\ell$-partition is an $\ell$-covering, \ref{item:cRPWDC-cdDP} implies~\ref{item:cRPWDC-cdDC}, and \ref{item:cRPWDC-cfDP} implies~\ref{item:cRPWDC-cfDC}. 
    \cref{dtofdisjointed} below says that \ref{item:cRPWDC-cdDC} implies~\ref{item:cRPWDC-cfDP}. Finally, \cref{Main} below says that \ref{item:cRPWDC-cfDC} implies~\ref{item:cRPWDC-cTPW}. 
\end{proof}

By definition, every ${(c,f)}$-disjointed $\ell$-covering is ${(c,f(1))}$-disjointed. 
The next lemma gives a qualitative converse to this.

\begin{lem}
    \label{dtofdisjointed}
    Let~${\ell,c,d\in \NN}$, and let~$\beta$ be a ${(c,d)}$-disjointed $\ell$-covering of a graph~$G$. 
    Then~$\beta$ is $(c,f)$-disjointed, where~${f(t) \coloneqq d t^c}$ for each~${t \in \NN}$.
\end{lem}

\begin{proof}
    Fix~${t \in \NN}$. 
    Let~${B_1,\dots, B_c \in \beta[t]}$. 
    Let~$X$ be a component of~${G-(B_1\cup\dots\cup B_c)}$. 
    For each $i\in \{1,\dots ,c\}$, let $\mathcal{B}_i$ be a set of at most $t$ elements of $\beta$ whose union is~$B_i$. 
    Let~${\mathcal{F} \coloneqq \mathcal{B}_1 \times \dots \times \mathcal{B}_c}$, and for each~${y = (A_1, \dots , A_c) \in \mathcal{F}}$, define~$Q_y$ as follows. 
    Let~$X_y$ the component of~${G - (A_1 \cup \dots \cup A_c)}$ containing~$X$. 
    Since~$\beta$ is $(c,d)$-disjointed, there exists~${Q_y \subseteq V(X_y)}$ of size at most~$d$ such that for every component~$Y$ of~${X_y-Q_y}$ there is some~${i \in \{1, \dots, c\}}$ such that~${V(Y) \cap N_G(A_i \setminus (A_1 \cup \dots \cup A_{i-1})) = \emptyset}$. 
    Now let~${Q \coloneqq \bigcup_{y \in \mathcal{F}} Q_y}$, and note that~${\abs{Q}\leq d\abs{\mathcal{F}} \leq dt^c}$. 
    
    Suppose for contradiction that for some component~$Y$ of~${X - Q}$ and each~${i \in \{1,\dots, c\}}$, there is a vertex~${b_i \in N_G(Y) \cap B'_i}$, where~${B'_i \coloneqq B_i \setminus (B_1 \cup \dots \cup B_{i-1})}$. 
    Let~${y = (A_1, \dots, A_c) \in \mathcal{F}}$ be such that~${(b_1, \dots, b_c) \in A_1 \times \dots \times A_c}$, and consider that component $Y'$ of~${X_y - Q_y}$ containing~$Y$. 
    By the definition of~$Q_y$, there is some~${i \in \{1,\dots, c\}}$ such that $Y'$ contains no neighbour of a vertex in $A_i \setminus (A_1 \cup \dots \cup A_{i-1})$. 
    In particular, all neighbours of vertices of~$Y$ are either vertices of~$Y'$ or neighbours of vertices of~${Y'}$, so~${b_i}$ is not a neighbour of any vertex of~$Y$, a contradiction.
\end{proof}

Now we prove that having a $(c,d)$-disjointed partition is necessary for bounded $c$-tree-partition-width.

\begin{lem}\label{disjointed}
    For all ${c,\ell\in\NN_0}$, every graph~$G$ with $c$-tree-partition-width~$\ell$ has a $(c,c\ell)$-disjointed $\ell$-partition. 
\end{lem}

\begin{proof}
    By assumption, $G$ has an $H$-partition ${\beta = (V_h \colon h \in V(H))}$ where~$H$ is a graph of treewidth at most~$c$ and~${\abs{V_h} \leq \ell}$ for all~$h$. 
    We first show that the singleton partition of~$H$ is $(c,c)$-disjointed. 
    Let~${v_1, \dotsc, v_c \in V(H)}$ and let~$X$ be a component of~${H-\{v_1,\dots, v_c\}}$. 
    Let ${(W_x \colon x \in V(T))}$ be a tree-decomposition of~$H$ where~${\abs{W_x} \leq c + 1}$ for all~${x \in V(T)}$. 
    We may assume that~${W_x \neq W_y}$ whenever~${x \neq y}$. 
    For each~${i \in \{1,\dots,c\}}$, let~$T_i$ be the subtree of~$T$ induced by~${\{x \in V(T) \colon v_i \in W_x\}}$. 
    
    First suppose that ${V(T_i) \cap V(T_j) = \emptyset}$ for some~${i, j \in \{1,\dots,c\}}$. 
    Let ${z \in V(T_i)}$ be the closest node (in~$T$) to~$T_j$. 
    Let~${Q \coloneqq W_z \cap X}$. 
    Note that~${Q \subseteq W_z \setminus \{v_i\}}$ so~${\abs{Q} \leq c}$. 
    Any path from~$v_i$ to~$v_j$ in~$H$ passes through~$W_z$, so each component of~${X - Q}$ is disjoint from~$N_H(v_i)$ or~$N_H(v_j)$. 
        
    Now assume that~${V(T_i) \cap V(T_j) \neq \emptyset}$ for all~${i, j\in \{1,\dots, c\}}$. 
    Let~$T_X$ be the subgraph of~$T$ induced by~${\{x \in V(T) \colon V(X) \cap W_x \neq \emptyset\}}$. 
    Since~$X$ is connected, $T_X$ is a subtree of~$T$. 
    Suppose that ${V(T_i) \cap V(T_X) = \emptyset}$ for some~$i$. 
    Since~${N_H(v_i) \subseteq \bigcup(W_x \colon x \in V(T_i))}$, it follows that~${N_H(v_i) \cap V(X) = \emptyset}$ and so we may take~${Q \coloneqq \emptyset}$ in this case. 
    Now assume that~${V(T_i) \cap V(T_X) \neq \emptyset}$ for all~${i \in \{1,\dots,c\}}$. 
    By the Helly property,
    ${\tilde{T} \coloneqq T_1 \cap \dotsb \cap T_c \cap T_X}$ is a non-empty subtree of~$T$. 
    For~${x \in V(\tilde{T})}$, we have~${\abs{W_x} \leq c + 1}$ and so~${W_x = \{v_1, \dotsc, v_c, u\}}$ for some~${u \in V(X)}$.
    First suppose that~${\abs{V(\tilde{T})} \geq 2}$. 
    Then there are adjacent ${x, y \in V(\tilde{T})}$ with ${W_x = \{v_1, \dotsc, v_c, u\}}$ and ${W_y = \{v_1, \dotsc, v_c, v\}}$ for~${u, v \in V(X)}$. 
    Since ${W_x \neq W_y}$, we have~${u \neq v}$ and thus there is no $(u,v)$-path in~${H - \{v_1, \dotsc, v_c\}}$, contradicting the connectedness of~$X$. 
    Hence~$\tilde{T}$ consists of a single vertex~$z$; thus ${W_z = \{v_1, \dotsc, v_c, u\}}$ for some~${u \in V(X)}$. 
    Let~${Q \coloneqq \{u\}}$ and consider a component~$Y$ of~${X - Q}$. 
    Let~$T_{Y}$ be the subtree of~$T$ induced by~${\{y\in V(T) \colon V(Y) \cap W_y \neq \emptyset\}}$. 
    Since~$T_{Y}$ is connected and does not contain~$z$, it is disjoint from some~$T_i$. 
    As above, ${N_H(v_i) \cap V(Y) = \emptyset}$, as required. 
        
    We have shown that the singleton partition of~$H$ is $(c,c)$-disjointed. 
    Now focus on~$G$. 
    By assumption, $\beta$ is an $\ell$-partition of~$G$. 
    Let~${V_{v_1}, \dotsc, V_{v_c}}$ be parts in~$\beta$, and let~$X$ be a component of~${G - (V_{v_1} \cup \dotsb \cup V_{v_c})}$. 
    Then~${X \subseteq \bigcup \{V_h \colon h \in X'\}}$ where~$X'$ is a component of~${H - \{v_1, \dotsc, v_c\}}$. 
    Since~$H$ is $(c,c)$-disjointed, there exists~${Q' \subseteq V(X')}$ of size at most~$c$ such that each component~${X' - Q'}$ is disjoint from some~${N_H(v_i)}$. 
    Let~${Q \coloneqq \bigcup \{V_h \colon h \in Q'\}}$, which has size at most~${c\ell}$. 
    Each component of~${X - Q}$ is disjoint from some~$N_G(V_{v_i})$.
\end{proof}

Note that $(c, f)$-disjointedness is preserved when restricting to a subgraph.

\begin{lem}
    \label{CoveringSubgraph}
    If~$\beta$ is a $(c,f)$-disjointed $\ell$-covering of a graph~$G$, then for every subgraph~$\tilde{G}$ of~$G$, the restriction ${\tilde{\beta} \coloneqq \{ B \cap V(\tilde{G}) \colon B \in \beta\}}$
    is a ${(c,f)}$-disjointed $\ell$-covering of~$\tilde{G}$.
\end{lem}

\begin{proof}
    Fix ${t \in \NN}$. 
    Let ${\tilde{B}_1, \dots, \tilde{B}_c \in \tilde{\beta}[t]}$ and let~$\tilde{X}$ be a component of~${\tilde{G}-(\tilde{B}_1\cup\dots \cup \tilde{B}_c)}$. 
    For each ${i \in \{1,\dots, c\}}$, there is a subset~${S_i \subseteq \beta}$ of size at most~$t$ such that~${\tilde{B}_i = \bigcup_{B \in S_i} (B\cap V(\tilde{G}))}$. 
    Let ${(B_1,\dots,B_c)\coloneqq (\bigcup S_1,\dots ,\bigcup S_c)}$, and let $\beta''$ be the $t\ell$-covering of $G$ given by $\beta\cup \{B_1,\dots,B_c\}$. 
    Let $X$ be the component of~${G-(B_1\cup \dots \cup B_c})$ which contains~$\tilde{X}$, and for each~${i \in \{1,\dots, c\}}$ let ${B'_i\coloneqq B_i\setminus (B_1\cup\dots \cup B_{i-1})}$.
    Since~$\beta$ is $(c,f)$-disjointed, there is a subset~$Q$ of~${V(X)}$ of size at most~${f(t)}$ such that each component of~${X-Q}$ disjoint from ${N_G(B'_i)}$ for some~${i \in \{1,\dots,c\}}$. 
    Let~${\tilde{Q}\coloneqq Q\cap V(\tilde{X})}$, and note that~${\abs{\tilde{Q}}\leq \abs{Q}\leq f(t)}$. 
    Each component of~${\tilde{X}-\tilde{Q}}$ is contained in a component of~${X - Q}$, and hence is disjoint from~${N_{\tilde{G}}(\tilde{B}_i\setminus (\tilde{B}_1\cup \dots \cup \tilde{B}_{i-1}))\subseteq N_G(B'_i)}$ for some~${i \in \{1,\dots, c\}}$. 
    Hence~$\tilde{\beta}$ is $(c,f)$-disjointed.
\end{proof}

The next lemma lies at the heart of the paper. 

\begin{lem}
    \label{Main}
    Let~${k,c,\ell \in \NN}$ and~${f \colon \NN \to \REALS^+}$. 
    For any graph $G$, if $\tw(G)<k$ and $G$ has a $(c,f)$-disjointed $\ell$-covering, then~$G$ has $c$-tree-partition-width ${\tpw_c(G) \leq \max\{ 12\ell k , 2 c \ell f(12 k)\}}$. 
\end{lem}

We prove \cref{Main} via the following induction hypothesis. 

\begin{lem}
    \label{MainLemma}
    Let ${k,c,\ell\in\NN}$ and let~${f \colon \NN \to \REALS^+}$. 
    Let~$G$ be a graph of treewidth less than~$k$ and let~${\beta \subseteq 2^{V(G)}}$ be a $(c,f)$-disjointed $\ell$-covering of~$G$. 
    Let ${S_1, \dotsc, S_{c - 1}, R \subseteq V(G)}$, where ${S_i \in \beta[12k]}$ for each~${i \in \{1, \dotsc, c - 1\}}$ and~${4 k\leq \abs{R} \leq f(12 k)}$. 
    Then there exists a $c$-tree-partition ${(V_x \colon x \in V(H))}$ of~$G$ of width at most~${W \coloneqq \max\{ 12\ell k , 2 c \ell f(12 k) \}}$, and there exists a $c$-clique ${\{x_1,\dots,x_{c-1},y\}}$ of~$H$ such that~${V_{x_i} = S_i\setminus(S_1\cup\dots\cup S_{i-1})}$ for each~${i \in \{1,\dots,c-1\}}$, and~${R \setminus (S_1 \cup \dots \cup S_{c - 1}) \subseteq V_y}$ with~${\abs{V_y} \leq 2 \ell (\abs{R} - 2 k)}$. 
\end{lem}

\begin{proof}
    We proceed by induction on~${\abs{V(G)}}$. 
    Let~${S \coloneqq S_1 \cup \dotsb \cup S_{c - 1}}$.
    
    \textbf{Case 0.}~${V(G) = R \cup S}$: 
    Let~$H$ be the complete graph on vertices~${x_1, \dotsc, x_{c - 1}, y}$. 
    Let ${V_{x_i} \coloneqq S_i \setminus (S_1 \cup \dotsb \cup S_{i - 1})}$ for each~$i$ and let~${V_y \coloneqq R}$. 
    Then~${(V_x \colon x \in V(H))}$ is a $c$-tree-partition of~$G$ with width at most~${W}$ and~${\abs{V_y} = \abs{R} \leq 2(\abs{R} - 2k) \leq 2 \ell(\abs{R} - 2k)}$. 
    From now on assume that~${G - (R \cup S)}$ is non-empty.
    
    \textbf{Case 1.}~${4 k \leq \abs{R}\leq 12 k}$:
    Since~$\beta$ is an $\ell$-covering, and~${\abs{R} \leq 12k}$, we can pick~${S_c \in \beta[12k]}$ such that~${R \subseteq S_c}$ and~${\abs{S_c} \leq \ell \abs{R} \leq 2 \ell (\abs{R} - 2k)}$. 
    
    
    Let~${G_1, \dotsc, G_a}$ be the connected components of~${G - (S \cup S_c)}$. 
    For each~${i \in \{1,\dots,c\}}$, let~${S'_i \coloneqq S_i \setminus (S_1 \cup \dotsb \cup S_{i - 1})}$. To complete this case, we first prove the following.
    
    \begin{claim}
        For each~${j \in \{1,\dots,a\}}$, the subgraph~${G[ V(G_j) \cup S \cup S_c ]}$ has a $c$-tree-partition ${(V_h^j \colon h \in V(H_j))}$ of width at most~$W$ such that there is a $c$-clique ${K = \{x_1, \dotsc, x_c\}}$ in~$H_j$, where~${S_i' = V^j_{x_i}}$ for each~${i \in \{1,\dots,c\}}$.
    \end{claim}

    \begin{proofofclaim}
        If~${\abs{V(G_j)} < 4 k}$, then take~$H_j$ to be the complete graph on vertices~${x_1, \dotsc, x_{c}, z}$ with the partition~${V^j_{x_i} \coloneqq S'_i}$ for~${i \in\{1,\dots,c\}}$ and~${V^j_{z} \coloneqq V(G_j)}$. 
        Then this gives us the desired $c$-tree-partition of~${G[V(G_j) \cup S \cup S_c]}$. 
        So assume~${\abs{V(G_j)} \geq 4k}$. 
        Note that $\beta[12k]$ is a $(c,f(12 k))$-disjointed $12k\ell$-covering containing~${S_1, \dots, S_c}$, so there is a subset~${Q'_j \subseteq V(G_j)}$ of size at most~${f(12 k)}$ and there is a partition~${\{A_1, \dotsc, A_c\}}$ of~${V(G_j - Q'_j)}$ such that each~$A_i$ is a union of vertex sets of components of~${G_j - Q'_j}$ that do not intersect~${N_{G}(S'_i)}$.
        Let~$Q_j$ be a set such that~${Q'_j \subseteq Q_j \subseteq V(G_j)}$ and~${4k \leq \abs{Q_j} \leq f(12 k)}$.
        As illustrated in \cref{DisjointedCovering2}, consider the subgraph 
        \begin{align*}
            F_i \coloneqq &\ G[A_i \cup Q_j \cup S_1 \cup \dotsb \cup S_{i - 1} \cup (S_{i + 1} \setminus S'_i) \cup \dotsb \cup (S_c \setminus S'_i)] \\
            = &\ G[A_{i} \cup Q_{j} \cup S'_{1} \cup \dotsb \cup S'_{i - 1} \cup S'_{i + 1} \cup \dotsb \cup S'_{c}].
        \end{align*}

        
        
        \begin{figure}[ht]
            \centering
            \includegraphics{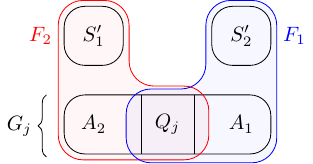}
\caption{The graphs $F_1$ and $F_2$ in the case $c=2$.}
            \label{DisjointedCovering2}
        \end{figure}
        
        By \cref{CoveringSubgraph}, the restriction of~$\beta$ to~${V(F_i)}$ is a $(c,f)$-disjointed $\ell$-covering of~$F_i$. 
        Apply induction to~$F_i$ with the sets~${S_1,\dotsc,S_{i - 1},S_{i + 1} \setminus S'_i,\dotsc,S_c\setminus S'_i}$ in place of the sets~${S_1,\dotsc,S_{c-1}}$ and the set~$Q_{j}$ in the place of~$R$.
        For each~${i \in \{1,\dotsc,c\}}$, this gives a graph~$L_i$ of treewidth at most~$c$ containing a $c$-clique~${\{x_1, \dotsc, x_{i-1}, x_{i+1},\dotsc, x_c, y\}}$ such that~$F_i$ has an $L_i$-partition~${( V^{j,i}_x \colon x \in V(L_i) )}$ with~${V^{j,i}_{x_m} = S'_m}$ for all~${m \in \{1,\dots, i-1,i+1,\dots, c\}}$ and~${Q_j \setminus ((S\cup S_c) \setminus S'_i)) \subseteq V_y^{j,i}}$ where 
        \begin{equation*}
            \abs{V^{j,i}_y} \leq 2 \ell (\abs{Q_j} - 2k) \leq 2 \ell f(12 k).
        \end{equation*}
        Let~$L_i^+$ be~$L_i$ together with a vertex~$x_i$ adjacent to the clique ${\{x_1, \dotsc, x_{i - 1}, x_{i + 1}, \dotsc, x_c, y\}}$. 
        So~${\tw(L^+_i)\leq c}$. 
        Set~${V^{j,i}_{x_i} \coloneqq S'_i}$. 
        Then~$L^+_i$ contains the $(c + 1)$-clique ${K^+ \coloneqq \{x_1, \dotsc, x_c, y\}}$ and ${(V^{j,i}_h \colon h \in V(L^+_i))}$ is an $L^+_i$-partition of~${G[A_i \cup Q_j \cup S \cup S_c]}$. 
        Now we may assume that~${V(L^+_1), \dotsc, V(L^+_c)}$ pairwise intersect in exactly the clique~$K^+$.
        Let~${H_j \coloneqq L_1^+\cup \dots\cup L_c^+}$. 
        Since each~$L_i^+$ has treewidth at most~$c$, so does~$H_j$. 
        For~${x \notin K^+}$, set~$V^j_x \coloneqq V^{j,i}_x$ for the unique~$i$ for which ${x \in V(L_i^+)}$, for ${i \in \{1, \dotsc, c \}}$ set ${V^j_{x_i} \coloneqq S'_i}$, and set ${V^j_y \coloneqq \bigcup_{i \in \{1, \dotsc, c\}} V^{j,i}_y}$. 
        Since~${\abs{V^{j}_y} \leq 2c \ell  f(12 k)}$, the partition ${( V^j_x \colon x \in V(H_j) )}$ has width at most~${W}$. 
        Setting~${K \coloneqq \{ x_1 \dotsc, x_c \}}$, the claim follows. 
    \end{proofofclaim}
    We may assume that ${V(H_1), \dotsc, V(H_a)}$ pairwise intersect in exactly the clique~$K$. 
    Let~$H \coloneqq {H_1\cup \dotsc \cup H_a}$. For~${x \in V(H)}$, setting~$V_x \coloneqq V^j_x$ if~${x \in V(H_j)}$ is well defined, and yields an $H$-partition $(V_h \colon h \in V(H))$ of~$G$. 
    Since each $H_i$ has treewidth at most~$c$, so does~$H$. 
    Let~${y \coloneqq x_c}$. 
    Then~${R \setminus S \subseteq S'_c = V_{y}}$, and, as noted at the start of this case, $\abs{V_{y}} = \abs{S'_c} \leq \abs{S_c} \leq 2 \ell (\abs{R} - 2k)$. 
    Hence, the width of this partition is at most~${W}$, as required. 
    
    \textbf{Case 2.}~${12 k < \abs{R} \leq f(12 k)}$: 
    Since~${\tw(G) < k}$, by the separator lemma of \citet[(2.6)]{RS-II},
    there is a partition~${(A, B, C)}$
    of~$V(G)$ with no edges between~$A$ and~$B$, where~${\abs{C} \leq k}$ and~${\abs{A \cap R}, \abs{B \cap R} \leq \tfrac{2}{3} \abs{R \setminus C}}$. 
    Let~${G_1 \coloneqq G[A \cup C]}$ and~${G_2 \coloneqq G[B \cup C]}$.
    Let~${R_1 \coloneqq (R \cap A) \cup C}$ and~${R_2 \coloneqq (R \cap B) \cup C}$. 
    Since~${\abs{R} \geq 12 k}$, 
    \begin{align*}
        \abs{R_1} & = \abs{A \cap R} + \abs{C} \leq \tfrac{2}{3} \abs{R} + k < \abs{R} \quad \textnormal{and}\\
        \abs{R_1} & \geq \abs{R} - \abs{B \cap R} \geq \abs{R} - \tfrac{2}{3} \abs{R \setminus C} \geq \tfrac{1}{3}\abs{R}\geq 4k.
    \end{align*}

    Hence, $4 k \leq \abs{R_1} \leq f(12k)$ and similarly $4 k \leq \abs{R_2} \leq f(12k)$. Also $\abs{V(G) - V(G_1)} = \abs{B} \geq \abs{R_2} - \abs{C} \geq 4k - k > 0$, so ${\abs{V(G_1)} < \abs{V(G)}}$ and likewise for~$G_2$.
    Fix~${j \in \{1, 2\}}$.
    Let~$\beta_j$ be the restriction of~$\beta$ to~${V(G_j)}$. 
    By \cref{CoveringSubgraph}, $\beta_j$ is a $(c, f)$-disjointed $\ell$-covering of~$G_j$. 
    Let~${S^j_i \coloneqq S_i \cap V(G_j)}$ for each~${i \in \{1,\dots,c-1\}}$; 
    note that each set~$S^j_i$ is a union of at most~$12k$ elements of~$\beta_j$. 
    
    Apply induction to~$G_j$ with~$S_i^j$ in place of~$S_i$ and~$R_j$ in place of~$R$. 
    Thus, there is a graph~$H_j$ of treewidth at most~$c$, a $c$-clique ${\{x_1, \dotsc, x_{c - 1}, y\}}$ of~$H_j$, and an $H_j$-partition ${(V^j_x \colon x \in V(H_j))}$ of~$G_j$ of width at most~$W$ such that~${V^j_{x_i} = S^j_i \setminus (S^j_1 \cup \dotsb \cup S^j_{i - 1})}$ for each~$i$ and ${R_j \setminus (S_{1}^j \cup \dotsb \cup S_{c - 1}^j) \subseteq V^j_{y}}$ with~${\abs{V_{y}} \leq 2 \ell (\abs{R_j} - 2k)}$. 
    We may assume that the intersection of~$V(H_1)$ and~$V(H_2)$ is equal to the clique~${K \coloneqq \{x_1, \dotsc, x_{c - 1}, y\}}$.
    Let~$H$ be the union of~$H_1$ and~$H_2$ and consider the $H$-partition~${(V_x \colon x \in V(H))}$ of~$G$ given by~${V_x \coloneqq V^j_x}$ for~${x \in V(H_j) \setminus V(H_{3-j})}$ and~${V_x \coloneqq V^1_x \cup V^2_x}$ for~${x \in K}$. 
    As~$H_1$ and~$H_2$ both have treewidth at most~$c$, so does~$H$. 
    
    Now ${\{x_1, \dotsc, x_{c - 1}, y\}}$ is a $c$-clique, ${V_{x_i} = S_{i} \setminus (S_1 \cup \dotsb \cup S_{i - 1})}$ and~${R \setminus S \subseteq V_y}$ with
    \begin{align*}
        \abs{V_y} 
        \;\leq\; \abs{V_{y}^1} + \abs{V_{y}^2} 
        \;\leq\; 2 \ell (\abs{R_1} + \abs{R_2} - 4k) 
         \;\leq\; 2 \ell (\abs{R} + 2 \abs{C} - 4k) 
         \;\leq\; 2 \ell (\abs{R} - 2k).
    \end{align*}
    The other parts do not change and so we have the desired $H$-partition of~$G$.
\end{proof}

We are now ready to prove \cref{Main}.

\begin{proof}[Proof of \cref{Main}]
    Let~$\beta$ be a $(c,f)$-disjointed $\ell$-covering of~$G$. 
    If~${\abs{V(G)} < 4k}$, then the trivial partition $\{V(G)\}$ satisfies the claim. 
    Otherwise~${\abs{V(G)} \geq 4k}$. 
    Let~${R \subseteq V(G)}$ with~${\abs{R} = 4k}$. 
    Let~${S_1,\dots,S_{c-1}}$ be arbitrary elements of~$\beta$. 
    Since~$\beta$ is an $\ell$-covering, ${\abs{S_i} \leq \ell \leq 12 \ell k}$ for each~$i$. 
    Now~\cref{Main} follows from \cref{MainLemma}. 
\end{proof}

\cref{Main,dtofdisjointed} imply the following result:

\begin{cor}
    \label{MainCorollary}
    Let~${k,c,d,\ell \in \NN}$. For any graph $G$, if $\tw(G)<k$ and~$G$ has a $(c,d)$-disjointed $\ell$-covering, then~$G$ has $c$-tree-partition-width ${\tpw_c(G) \leq 2 c d \ell (12 k)^c}$. 
\end{cor}

Note that the singleton partition of any graph with maximum degree~$\Delta$ is $(1,\Delta)$-disjointed.  So \cref{MainCorollary} with ${c=\ell=1}$ and ${d=\Delta}$ implies \cref{TreePartitionWidthDegree} (even with the same constant 24). Indeed, the proof of \cref{Main} in the case of graphs with bounded degree is equivalent to the proof of \cref{TreePartitionWidthDegree}. 

To conclude this section, \cref{disjointed,Main} imply the following characterisation of underlying treewidth. 

\begin{thm}
\label{UnderlyingTreewidthDisjointedPartition}
The underlying treewidth of a graph class~$\GG$ equals the minimum~${c \in \NN_0}$ such that, for some function~${g \colon \NN \to \NN}$, every graph~${G \in \GG}$ has a $(c,g(\tw(G)))$-disjointed $g(\tw(G))$-partition. 
\end{thm}

\section{Lower Bounds}
\label{LowerBounds}

We now define two graphs that provide lower bounds on the underlying treewidth of various graph classes. Recall that $\widehat{\ell G}$ is the graph obtained from $\ell$ disjoint copies of a graph $G$ by adding one dominant vertex. For $c,\ell\in\NN$, define graphs $G_{c, \ell}$ and $C_{c, \ell}$ recursively as follows. First, $G_{1, \ell} \coloneqq P_{\ell + 1}$ is the path on $\ell + 1$ vertices, and $C_{1, \ell} \coloneqq K_{1, \ell}$ is the star with $\ell$ leaves. Further, for $c \geq 2$, let $G_{c, \ell} \coloneqq \widehat{\ell\, G_{c - 1, \ell}}$ and $C_{c, \ell} \coloneqq \widehat{\ell\, C_{c - 1, \ell}}$. Note that $C_{c, \ell}$ is the closure of the rooted complete $\ell$-ary tree of height $c$. Here, the \defn{closure} of rooted tree $T$ is the graph $G$ with $V(G)=V(T)$ and $vw\in E(G)$ whenever $v$ is an ancestor of $w$ or $w$ is an ancestor of~$v$. 


The next lemma collects together some useful and well-known properties of $G_{c, \ell}$ and $C_{c, \ell}$.

\begin{lem}\label{StandardExamples}
    For all $c,\ell\in\NN$, 
    \begin{enumerate}
        [label=\textnormal{(\roman*)}]
        \item \label{item:SE-1} $\tw(G_{c, \ell}) = \tw(C_{c, \ell}) = c$\textnormal{;}
        \item \label{item:SE-2} for any $\ell$-partition of $G \in \{ G_{c, \ell}, C_{c, \ell} \}$, there is a $(c + 1)$-clique in $G$ whose vertices are in distinct parts\textnormal{;}
        \item \label{item:SE-3} $G_{c, \ell}$ and $C_{c, \ell}$ both have $(c - 1)$-tree-partition-width greater than $\ell$\textnormal{;}
        \item \label{item:SE-4} $G_{2, \ell}$ is outerplanar, $G_{3, \ell}$ is planar, $G_{4, \ell}$ is linklessly embeddable\textnormal{;}
        \item \label{item:SE-5} $G_{c, \ell}$ is $K_{c, \max\{c, 3\}}$-minor-free\textnormal{;}
        \item \label{item:SE-6} $C_{c, \ell}$ does not contain $P_{4}$ as an induced subgraph\textnormal{;}
        \item \label{item:SE-7} $C_{c,\ell}$ does not contain $P_{n}$ as a subgraph for $n\geq 2^{c+1}$. 
    \end{enumerate}
\end{lem}

\begin{proof}
    Since~${\tw(\widehat{\ell\, G}) = \tw(G)+1}$ for any graph~$G$ and~${\ell \in \NN}$, part~\ref{item:SE-1} follows by induction. 
    
    We establish~\ref{item:SE-2} by induction on~$c$. 
    In the case~${c = 1}$, every $\ell$-partition of~$P_{\ell + 1}$ or~$K_{1, \ell}$ contains an edge whose endpoints are in different parts, and we are done. 
    Now assume the claim for~${c - 1}$ (${c \geq 2}$) and let~${G \in \{ G_{c - 1, \ell}, C_{c - 1, \ell}\}}$. 
    Consider an $\ell$-partition of~${\widehat{\ell\, G}}$. 
    At most~${\ell - 1}$ copies of~$G$ contain a vertex in the same part as the dominant vertex~$v$. 
    Thus, some copy~$G_0$ of~$G$ contains no vertices in the same part as~$v$. 
    By induction, $G_0$ contains a $c$-clique~$K$ whose vertices are in distinct parts. 
    Since~$v$ is dominant, ${K \cup \{v\}}$ satisfies the induction hypothesis. 
    
    Let~${G \in \{G_{c, \ell}, C_{c, \ell} \}}$. 
    Consider an $H$-partition of~$G$ of width at most~$\ell$. By~\ref{item:SE-2}, $G$ contains a ${(c+1)}$-clique whose vertices are in distinct parts. 
    So~${\omega(H) \geq c+1}$, implying~${\tw(H) \geq c}$. 
    This establishes~\ref{item:SE-3}. 
    
    Observe that~$G_{2, \ell}$ is outerplanar (called a \defn{fan} graph). 
    The disjoint union of outerplanar graphs is outerplanar and the graph obtained from any outerplanar graph by adding a dominant vertex is planar; 
    thus~$G_{3, \ell}$ is planar. 
    The disjoint union of planar graphs is planar, and the graph obtained from any planar graph by adding a dominant vertex is linklessly embeddable; 
    thus~$G_{4, \ell}$ is linklessly embeddable. T
    his proves~\ref{item:SE-4}. 
    
    We next show that $G_{c, \ell}$ is $K_{c, \max\{c, 3\}}$-minor-free. 
    $G_{1, \ell}$ is a path and so has no $K_{1, 3}$-minor. 
    $G_{2, \ell}$ is outerplanar and so has no $K_{2, 3}$-minor. 
    Let~${c \geq 3}$ and assume the result holds for smaller~$c$. 
    Suppose that~$G_{c, \ell}$ contains a $K_{c,c}$-minor. 
    Since~$K_{c,c}$ is 2-connected, some copy of~$G_{c-1, \ell}$ in~$G_{c, \ell}$ contains a $K_{c-1, c}$-minor. 
    This contradiction establishes~\ref{item:SE-5}. 
    
    We show that~$C_{c, \ell}$ contains no induced~$P_{4}$ by induction on~$c$. 
    First, ${C_{1, \ell} = K_{1, \ell}}$ does not contain~$P_{4}$. 
    Next, suppose that~$C_{c-1,\ell}$ does not contain an induced~$P_{4}$. 
    $P_{4}$ does not have a dominant vertex and so any induced~$P_{4}$ in~$C_{c, \ell}$ must lie entirely within one copy of~$C_{c-1,\ell}$. 
    In particular, $C_{c, \ell}$ does not contain an induced~$P_{4}$. This proves~\ref{item:SE-6}. 
    
    Finally, \citet{Sparsity} proved~\ref{item:SE-7}. 
\end{proof}

The underlying treewidth of the class of graphs of treewidth at most $k$ is obviously at most $k$.
\Cref{StandardExamples}~\ref{item:SE-1} and~\ref{item:SE-3} imply the following.

\begin{cor}
    \label{TreewidthUnderlyingTreewidth}
    The underlying treewidth of the class of graphs of treewidth at most~$k$ equals~$k$.
\end{cor}


\begin{cor}
\label{StandardExampleUnderlyingTreewidth}
    The classes $\{G_{c,\ell} \colon c,\ell \in \NN\}$ and $\{C_{c, \ell} \colon c, \ell \in \NN\}$ both have unbounded underlying treewidth. 
\end{cor}

\begin{proof}
    Suppose that ${\{G_{c,\ell} \colon c,\ell \in \NN\}}$ has underlying treewidth~$b$. 
    Thus, for some function~$f$, for all~${c, \ell \in \NN}$, we have~${\tpw_b(G_{c, \ell}) \leq f(\tw(G_{c,\ell})) = f(c)}$. 
    In particular, with~${c \coloneqq b + 1}$ and~${\ell \coloneqq f(c)}$, we have~${\tpw_{c-1}(G_{c,\ell}) \leq \ell}$, which contradicts \cref{StandardExamples}~\ref{item:SE-3}. 
    The proof for~${\{C_{c,\ell} \colon c,\ell \in \NN\}}$ is analogous.
\end{proof}

The graphs~$G_{c, \ell}$ and~$C_{c, \ell}$ are common in the graph-theory literature, and are particularly important for clustered colouring \citep{NSSW19,vdHW18,KMRV97,EJ14}, as we now explain. 
In an (improperly) vertex-coloured graph, a \defn{monochromatic component} is a connected component of the subgraph induced by all the vertices of one colour. 
A graph~$G$ is \defn{$c$-colourable with clustering~$\ell$} if each vertex can be assigned one of~$c$ colours such that each monochromatic component has at most~$\ell$ vertices. 
The \defn{clustered chromatic number} of a graph class~$\GG$
is the minimum~${c \in \NN}$ such that, for some~${\ell \in \NN}$, every graph in~$\GG$ has a $c$-colouring with clustering~$\ell$. 
See \citep{WoodSurvey} for a survey on clustered graph colouring. 
Note that a graph~$G$ is $c$-colourable with clustering~$\ell$ if and only if~$G$ is contained in~${H \boxtimes K_\ell}$ for some graph~$H$ with~${\chi(H) \leq c}$. 

Consider a graph class~$\GG$ with underlying treewidth~$c$, where every graph in~$\GG$ has treewidth at most~$k$. 
Thus every graph~${G \in \GG}$ is contained in~${H \boxtimes K_\ell}$ for some graph~$H$ with~${\tw(H) \leq c}$, where~${\ell \coloneqq \max\{f(0),\dots,f(k)\}}$ and~$f$ is from the definition of underlying treewidth. 
Since~${\chi(H) \leq \tw(H)+1}$ for every graph~$H$, it follows that~$\GG$ has clustered chromatic number at most~${c+1}$. 
This means that lower bounds on the clustered chromatic number of a graph class with bounded treewidth provide lower bounds on the underlying treewidth of the class. 
In particular, it is known that~$G_{c,\ell}$ and~$C_{c,\ell}$ are not $c$-colourable with clustering~$\ell$ (see \citep{WoodSurvey}), which implies \cref{StandardExamples}~\ref{item:SE-3} by the above connection. 
See \citep{NSSW19} for more examples of graphs~$G$ that are not $c$-colourable with clustering~$\ell$, implying~${\tpw_{c-1}(G) > \ell}$.

\section{Excluding a Minor}
\label{ExcludedMinor}

This section uses disjointed partitions to determine the underlying treewidth of several minor-closed classes of interest.

The next definition enables $K_t$-minor-free graphs and $K_{s,t}$-minor-free graphs to be handled simultaneously. For~${s,t\in\NN}$, let \defn{$\KK_{s,t}$} be the class of graphs~$G$ for which there is a partition~${\{A,B\}}$ of~${V(G)}$ such that~${\abs{A} = s}$ and~${\abs{B} = t}$; 
${vw \in E(G)}$ for all~${v \in A}$ and~${w \in B}$; and~${G[B]}$ is connected. 
Obviously, every graph in~$\KK_{s,t}$ contains~$K_{s,t}$. 
Similarly, we obtain~$K_t$ as a minor of any~${G \in \KK_{t-2,t}}$ by contracting a matching between~$A$ and~$B$ of size~${t-2}$ whose end-vertices are distinct from the end-vertices of some edge of~${G[B]}$. 

\begin{lem}
    \label{AssignTrick}
    Let~$G$ be a graph with no minor in~$\KK_{s,t}$. 
    Assume~${\{A,B\}}$ is a partition of~${V(G)}$ such that~${G[B]}$ is connected and every vertex in~$B$ has at least~$s$ neighbours in~$A$. 
    Then~${\abs{B} \leq \delta \abs{A}}$ for some~${\delta = \delta(s,t)}$. 
\end{lem}

\begin{proof}
    The following proof technique is well-known~\citep{OOW19,LW1}. 
    Assign vertices in~$B$ to pairs of vertices in~$A$ as follows. Initially, no vertices in~$B$ are assigned. 
    If there is an unassigned vertex~${v \in B}$ adjacent to distinct vertices~${x,y \in A}$ and no vertex in~$B$ is already assigned to~${\{x, y\}}$, then assign~$v$ to~${\{x,y\}}$. 
    Repeat this operation until no more vertices in~$B$ can be assigned. 
    Let~$B_1$ and~$B_2$ be the sets of assigned and unassigned vertices in~$B$, respectively. 
    Let~$G'$ be the graph obtained from~${G[A\cup B_1]}$ by contracting~$vx$ into~$x$, for each vertex~${v \in B_1}$ assigned to~${\{x, y\}}$. 
    Thus~$G'$ is a minor of~$G$ with~${V(G') = A}$ and~${\abs{E(G')} \geq \abs{B_1}}$. 
    For each vertex~${v \in B_2}$, the set~${N_G(v) \cap A}$ is a clique in~$G'$ (otherwise~$v$ would have been assigned to a pair of non-adjacent neighbours) of size at least~$s$. 
    For each $s$-clique~$C$ in~$G'$, there are at most~${t-1}$ vertices~${v \in B_2}$ with~${C \subseteq N_G(v)}$, otherwise~$G$ has a minor in~$\KK_{s,t}$ (since~${G[B]}$ is connected). 
    Since~$G'$ has no minor in~$\KK_{s,t}$, we have~${\abs{E(G')} \leq \delta \abs{A}}$ for some~${\delta = \delta(s,t)}$; see \citep{Thomason84,Thomason01,Kostochka84,Kostochka82,KP12,KP10,KP08,KO05,ReedWood16} for explicit bounds on~$\delta$. 
    Thus~${\abs{B_1} \leq \delta \abs{A}}$. 
    Moreover, $G'$ is ${2 \delta}$-degenerate. 
    Every $d$-degenerate graph on~$n$ vertices has at most~${\binom{d}{s - 1} n}$ cliques of order~$s$~\citep[Lemma 18]{Wood16}. 
    Thus~${\abs{B_2} \leq (t - 1)\binom{2\delta}{s - 1}\abs{A}}$. 
    Hence~${\abs{B} = \abs{B_1} + \abs{B_2} \leq (\delta + (t - 1) \binom{2\delta}{s - 1})\abs{A}}$.
\end{proof}

The following Erd\H{o}s--P\'{o}sa type result is useful for showing disjointedness. It follows from well-known results in the literature; see \cref{Omitted} for details.

\begin{restatable}{lem}{EPtreewidth}\label{EP-treewidth}
    Let $\HH$ be a set of connected subgraphs of a graph $G$. Then, for every non-negative integer $\ell$, either there are $\ell + 1$ vertex-disjoint graphs in $\HH$ or there is a set $Q \subseteq V(G)$ of size at most $\ell(\tw(G) + 1)$ such that $G - Q$ contains no graph of $\HH$.
\end{restatable}

\begin{lem}
    \label{XstMinorFreeTreePartition}
    For fixed~${s,t \in \NN}$, every graph~$G$ with no minor in~$\KK_{s,t}$ and of treewidth~$k$ has $s$-tree-partition-width~${O(k^2)}$. 
\end{lem}

\begin{proof}
    By \cref{Main} it suffices to show that the singleton partition of~$G$ is $(s,f)$-disjointed, where ${f(n) \coloneqq \delta sn (k + 1)}$ and $\delta \coloneqq \delta(s,t)$ from \cref{AssignTrick}. 
    Let~${S_1,\dots,S_s}$ be subsets of~${V(G)}$ of size at most~$n$, let~${S \coloneqq S_1 \cup \dots \cup S_s}$, and for each~${i \in \{1,\dots, s\}}$ let~${S'_i \coloneqq S_i \setminus (S_1 \cup \dots \cup S_{i-1})}$.
    Let~$X$ be a connected component of~${G-S}$. 
    Let~$\HH$ be the set of connected subgraphs~$H$ of~$X$ such that~${H \cap N(S'_i) \neq \emptyset}$ for each~${i \in \{1,\dots,s\}}$. 
    Say~$\RR$ is a maximum-sized set of pairwise disjoint subgraphs in~$\HH$. 
    We may assume that~${\bigcup \{ V(R) \colon R \in \RR \} = V(X)}$. 
    Let~$X'$ be the graph obtained from~${G[S \cup V(X)]}$ by contracting each subgraph~${R \in \RR}$ into a vertex~$v_R$. 
    So~${V(X') = S\cup\{v_R \colon R\in\RR\}}$. 
    Since~$X$ is connected, ${\{v_R \colon R \in \RR\}}$ induces a connected subgraph of~$X'$. 
    By construction, in~$X'$, each vertex~$v_R$ has at least~$s$ neighbours in~$S$. 
    By \cref{AssignTrick}, ${\abs{\RR} \leq \delta \abs{S}}$. 
    By \cref{EP-treewidth}, there is a set~${Q \subseteq V(X)}$ of size at most~${\delta \abs{S} (k + 1) \leq f(n)}$ such that~${X-Q}$ contains no graph in~$\HH$. 
    Thus each component~$Y$ of~${X-Q}$ satisfies~${V(Y) \cap N_G(S'_i) = \emptyset}$ for some~${i \in \{1,\dots,s\}}$. 
    Hence, the singleton partition of~$G$ is $(s, f)$-disjointed. 
\end{proof}

We now determine the underlying treewidth of $K_t$- and $K_{s,t}$-minor-free graphs. 

\begin{thm}
\label{NoKtMinorUnderlyingTreewidth}
For fixed $t\in\NN$ with $t\geq 2$, the underlying treewidth of the class of $K_t$-minor-free graphs equals $t-2$. In particular, every $K_t$-minor-free graph of treewidth $k$ has $(t-2)$-tree-partition-width $O(k^2)$. 
\end{thm}

\begin{proof}
Since $K_t$ is a minor of every graph in $\KK_{t-2,t}$, \cref{XstMinorFreeTreePartition} implies that every $K_t$-minor-free graph of treewidth $k$ has $(t-2)$-tree-partition-width $O(k^2)$. Thus the underlying treewidth of the class of $K_t$-minor-free graphs is at most $t-2$. Suppose that it is at most $t-3$. Thus, for some function $f$, every $K_t$-minor-free graph $G$ has $(t-3)$-tree-partition-width at most $f(\tw(G))$. Let $\ell\coloneqq f(t-2)$. The graph $G_{t - 2, \ell}$ in \cref{StandardExamples} has treewidth $t-2$ and is thus $K_t$-minor-free. However, by \cref{StandardExamples}, 
$\tpw_{t-3}(G_{t - 2, \ell})> \ell = f(t - 2) = f(\tw(G_{t - 2, \ell}))$, which is the required contradiction.
\end{proof}

\begin{thm}
    \label{NoKstMinorUnderlyingTreewidth}
    For fixed~${s,t \in \NN}$ with~${t \geq \max\{s,3\}}$, the underlying treewidth of the class of $K_{s,t}$-minor-free graphs equals~$s$. 
    In particular, every $K_{s,t}$-minor-free graph~$G$ of treewidth~$k$ has $s$-tree-partition-width~${O(k^2)}$.
\end{thm}

\begin{proof}
    Since~$K_{s,t}$ is a subgraph of every graph in~$\KK_{s,t}$,
    \cref{XstMinorFreeTreePartition} implies that every $K_{s,t}$-minor-free graph~$G$ of treewidth~$k$ has $s$-tree-partition-width~${O(k^2)}$. 
    Thus the underlying treewidth of the class of $K_{s,t}$-minor-free graphs is at most~$s$. 
    Suppose that it is at most~${s - 1}$. 
    Thus for some function~$f$, every $K_{s,t}$-minor-free graph~$G$ satisfies~${\tpw_{s-1}(G) \leq f(\tw(G))}$. 
    Let~${\ell \coloneqq f(s)}$. 
    The graph~$G_{s,\ell}$ given by \cref{StandardExamples} has treewidth~$s$ and is $K_{s,\max\{s,3\}}$-minor-free, implying it is $K_{s,t}$-minor-free (since~${t \geq \max\{s,3\}}$). 
    However, by \cref{StandardExamples}, ${\tpw_{s - 1}(G_{s, \ell}) > \ell = f(s) = f(\tw(G_{s, \ell}))}$, which is the required contradiction. 
\end{proof}

For~${t \leq 2}$, \cref{NoKstMinorUnderlyingTreewidth} is improved as follows:

\begin{prop}
    For~${s,t \in \NN}$ with~${s \leq t \leq 2}$, the underlying treewidth of the class of $K_{s,t}$-minor-free graphs equals~${s-1}$. 
\end{prop}

\begin{proof}
    Every $K_{1,1}$-minor-free graph~$G$ has no edges, and so~${\tpw_0(G) \leq 1}$. 
    Every $K_{1,2}$-minor-free graph~$G$ has at most one edge in each component, and so~${\tpw_0(G) \leq 2}$. 
    Each block of any $K_{2,2}$-minor-free graph~$G$ is a triangle, an edge or an isolated vertex; so~${\tpw_1(G) \leq 2}$. 
\end{proof}

Since planar graphs are $K_5$- and $K_{3,3}$-minor-free, 
\cref{NoKtMinorUnderlyingTreewidth} or \cref{NoKstMinorUnderlyingTreewidth} imply the next result (where the lower bound holds since the graph~$G_{3,\ell}$ in \cref{StandardExamples} is planar).

\begin{thm}
    \label{planarunderlying}
	The underlying treewidth of the class of planar graphs equals~$3$. In particular, every planar graph of treewidth~$k$ has $3$-tree-partition-width $O(k^2)$. 
\end{thm}

It follows from Euler's formula that every graph with Euler genus at most~$g$ is $K_{3, 2g + 3}$-minor-free. 
Thus \cref{StandardExamples,NoKstMinorUnderlyingTreewidth} imply the following. 

\begin{thm}
    \label{SurfaceProduct}
    The underlying treewidth of the class of graphs embeddable on any fixed surface~$\Sigma$ equals~$3$. 
    In particular, every graph embeddable in~$\Sigma$ and of treewidth~$k$ has $3$-tree-partition-width~${O(k^2)}$. 
\end{thm}

Since every linklessly embeddable graph is $K_6$-minor-free and $K_{4,4}$-minor-free~\citep{RST95},
\cref{NoKtMinorUnderlyingTreewidth} or \cref{NoKstMinorUnderlyingTreewidth} imply the next result, where the lower bound follows from \cref{StandardExamples}.

\begin{thm}
    \label{LinklessUnderlying}
    The underlying treewidth of the class of linklessly embeddable graphs equals~$4$. 
    In particular, every linklessly embeddable graph of treewidth~$k$ has $4$-tree-partition-width~${O(k^2)}$.
\end{thm}

It is an open problem to determine the underlying treewidth of a given minor-closed class~$\GG$. 
It is possible that the clustered chromatic number of $\GG$ equals the underlying treewidth of~$\GG$ plus~1. 
This is true for any minor-closed class with underlying treewidth at most~1 by results of \citet{NSSW19} and \citet{DO96}; see \cref{MinorClosedClass0,MinorClosedClass1}. 
See \citep{NSSW19} for a conjectured value of the clustered chromatic number of~$\GG$. 
It follows from a result of \citet{DDOSRSV04} that every minor-closed graph class~$\GG$ with underlying treewidth~$c$ has clustered chromatic number at most~${2(c+1)}$; see \cref{ClusteredApplication}.

\section{Excluding a Topological Minor}
\label{TopoMinors}

This section studies the underlying treewidth of graphs classes defined by an excluded topological minor. We conclude that a monotone graph class has bounded underlying treewidth if and only if it excludes some fixed topological minor.

\begin{thm}
    \label{TopoMinorPartition}
    For every fixed multigraph~$X$ with~$p$ vertices, every $X$-topological minor-free graph~$G$ of treewidth~$k$ has $p$-tree-partition-width~${O(k^2)}$. 
\end{thm}

\begin{proof}
    By \cref{Main} it suffices to show that the singleton partition of~$G$ is $(p,f)$-disjointed, where $f(n) \in O(kn)$. To this end, 
    let~${S_1,\dots,S_p}$ be subsets of~${V(G)}$ of size at most~$n$, let~${S \coloneqq S_1 \cup \dots \cup S_p}$, and for each~${i \in \{1, \dots, p\}}$ let~${S'_i \coloneqq S_i \setminus (S_1 \cup \dots \cup S_{i-1})}$.  
    We may assume that~${V(X) = \{1,\dots,p\}}$. 
    Let~$\HH$ be the set of connected subgraphs~$H$ of~${G-S}$ such that~${V(H) \cap N_G(S'_i) \neq \emptyset}$ for each~${i \in \{1,\dots,p\}}$. 
    
    Consider any set~$\JJ$ of pairwise vertex-disjoint graphs in~$\HH$ of maximum size. 
    Assign subgraphs in~$\JJ$ to pairs of vertices in~$S$ as follows. 
    Initially, no subgraphs in~$\JJ$ are assigned. 
    If there is an unassigned subgraph~${H \in \JJ}$ adjacent to vertices~${x \in S'_i}$ and~${y \in S'_j}$, for some distinct~${i,j \in \{1,\dots,p\}}$, and no subgraph in~$\JJ$ is already assigned to~${\{x,y\}}$, then assign~$H$ to~${\{x,y\}}$. 
    Repeat this operation until no more subgraphs in~$\JJ$ can be assigned. 
    
    Let~$\JJ_1$ and~$\JJ_2$ be the sets of assigned and unassigned subgraphs in~$\JJ$ respectively. 
    Let~$G'$ be the graph obtained from~$G$ as follows: 
    for each~${H \in \JJ_1}$ assigned to~${\{x,y\}}$, contract an $(x,y)$-path through~$H$ down to a single edge~$xy$. 
    Delete any remaining vertices not in~$S$. 
    Thus~$G'$ is a topological minor of~$G$ with~${V(G') = S}$ and~${\abs{E(G')} \geq \abs{\JJ_1}}$. 
    Consider a subgraph~${H \in \JJ_2}$. 
    Since~${H \in \HH}$ there are neighbours~${v_1,\dots,v_p}$ of~$H$ with~${v_i \in S'_i}$ for each~${i \in \{1,\dots,p\}}$. 
    Since~${H \in \JJ_2}$, the set~${\{v_1,\dots,v_p\}}$ is a clique in~$G'$ (otherwise~$H$ could have been assigned to some non-adjacent~$v_i$ and~$v_j$). 
    Charge~$H$ to~${(v_1,\dots,v_p)}$. 
    
    Suppose that there is a set~$\JJ'$ of at least~$\abs{E(X)}$ subgraphs in~$\JJ_2$ charged to some clique~${(v_1,\dots,v_p)}$ in~$G'$ with~${v_i \in S'_i}$. 
    Label~${\JJ' = \{H_e \colon e \in E(X)\}}$. 
    For~${e = ij \in E(X)}$, let~$P_e$ be a $(v_i,v_j)$-path through~$H_e$. 
    Thus~${\{v_i P_e v_j \colon e = ij \in E(X)\}}$ defines an $X$-topological minor in~$G$ with branch vertices~${v_1, \dotsc, v_p}$. 
    This contradiction shows that there are fewer than~$\abs{E(X)}$ subgraphs in~$\JJ_2$ charged to each clique ${(v_1, \dotsc, v_p)}$ in~$G'$ with~${v_i \in S'_i}$.
    
    Since~$G'$ is $X$-topological-minor-free, ${\abs{E(G')} \leq \gamma \abs{S}}$ for some~${\gamma = \gamma(X)}$; see \citep{BT98,TW05} for explicit bounds on~$\gamma$. 
    Thus~${\abs{\JJ_1} \leq \gamma \abs{S}}$. 
    Moreover, $G'$ is $2\gamma$-degenerate. 
    Every $d$-degenerate graph on~$n$ vertices has at most~${\binom{d}{p-1}n}$ cliques of order~$p$ \citep[Lemma 18]{Wood16}. 
    Thus~${\abs{\JJ_2} < \abs{E(X)} \binom{2\gamma}{p-1}\abs{S}}$. 
    Hence~${\abs{\JJ} = \abs{\JJ_1} + \abs{\JJ_2} < (\gamma + \abs{E(X)} \binom{2\gamma}{p-1}) \abs{S}}$. 
    Define
    \begin{equation*}
        f(n) \coloneqq (k + 1) \left(\gamma + \abs{E(X)} \tbinom{2\gamma}{p - 1}\right)np.
    \end{equation*}
    Since~${\abs{S} \leq pn}$, by \cref{EP-treewidth}, there is a set~${Q \subseteq V(G-S)}$ of size at most~${f(n)}$ such that no subgraph of~${G-(S\cup Q)}$ is in~$\HH$. 
    So every component~$Y$ of~${G-(S\cup Q)}$ satisfies~${Y \cap N_{G}(S'_i) = \emptyset}$ for some~${i \in \{1, \dots, p\}}$. 
    Hence the singleton partition of~$G$ is $(p,f)$-disjointed.
\end{proof}

\begin{thm}
    \label{TopoMinorUnderlyingTreewidth}
    The underlying treewidth of the class of $K_t$-topological-minor-free graphs equals~${t-2}$ if~${t \in \{2,3,4\}}$ and equals~$t$ if~${t \geq 5}$. 
\end{thm}

\begin{proof}
    For~${t \in \{2,3,4\}}$, a graph is $K_t$-topological-minor-free if and only if it has treewidth at most~${t-2}$. 
    So the result follows from \cref{TreewidthUnderlyingTreewidth}. 
    Now assume that~${t \geq 5}$. 
    
    \cref{TopoMinorPartition} implies that the underlying treewidth of the class of $K_t$-topological-minor-free graphs is at most~$t$. Suppose for the sake of contradiction that it is at most~${t-1}$. 
    That is, there is a function~$f$ such that every $K_t$-topological-minor-free graph~$G$ has $(t-1)$-tree-partition-width at most~${f(\tw(G))}$. 
    Let~${\ell \coloneqq f(t)}$. 
    Let~$G_1$ be the graph~$G_{t-2,\ell}$ from \cref{StandardExamples}. So~${\tw(G_1) = t-2}$, implying~$G_1$ is $K_t$-topological-minor-free, and for every $H$-partition of~$G_{1}$ with width~$\ell$, there is a ${(t - 1)}$-clique in~$G_1$ whose vertices are in distinct parts. 
    
    Let~${n \gg t,\ell}$ be a sufficiently large integer as detailed below. 
    Let~$G_2$ be the graph with vertex-set ${\{v_1,\dots,v_{t-1}\}\cup\{x_1,\dots,x_{n}\}}$
    where~${\{v_1,\dots,v_{t-1}\}}$ is a clique, ${x_1, \dots, x_{n}}$ is a path, and~${v_i x_j}$ is an edge whenever~${j = a(t-1)+i}$ for some~${a \in \NN_0}$. 
    Each subpath ${x_{a(t-1)+1},\dots,x_{a(t-1)+(t-1)}}$ is called a \defn{clump}. 
    By construction, each vertex~$v_i$ has one neighbour in every clump. 
    Since each vertex~$x_j$ has degree at most~3, the graph~$G_2$ has exactly~${t-1}$ vertices of degree at least~4. 
    Thus~$G_2$ contains no $K_t$-topological minor (since~${t \geq 5}$). 
    Let~${X_j \coloneqq \{v_1,\dots,v_{t-1},x_j,x_{j+1}\}}$. 
    Then ${(X_1,X_2,\dots,X_{n-1})}$ is a path-decomposition of~$G_2$ with width~$t$; thus~${\tw(G_2) \leq t}$. 
    
    Let~$G$ be obtained by pasting a copy of~$G_2$ onto each ${(t-1)}$-clique~$C$ of~$G_1$, where the vertices~${v_1,\dots,v_{t-1}}$ in~$G_2$ are identified with~$C$. 
    Since each of~$G_1$ and~$G_2$ have treewidth at most~$t$, so too does~$G$. 
    Since each of~$G_1$ and~$G_2$ contain no $K_t$-topological-minor, $G$ contains no $K_t$-topological minor. 
    Thus~$G$ has a ${(t-1)}$-tree-partition~${\PP = (V_h \colon h \in V(H))}$ of width~$\ell$.
    
    Since~$G_1$ is a subgraph of~$G$, there is a ${(t-1)}$-clique~$C$ in~$G_1$ whose vertices are in distinct parts of~$\PP$. 
    Let~$C'$ be the set of parts in~$\PP$ intersecting~$C$. 
    Thus~${\abs{C'} = \abs{C} = t-1}$. 
    Let~$P_0$ be the path in the copy of~$G_2$ pasted onto~$C$ in~$G$. 
    At most~${(t-1)(\ell-1)}$ of the vertices in~$P_0$ are in the parts of~$\PP$ in~$C'$. 
    Thus, for~${n \gg t, \ell}$, there is a subpath~$P_1$ of~$P_0$ that contains at least~${(t-1)\ell+2}$ clumps, and no vertex of~$P_1$ is in the parts of~$\PP$ in~$C'$. 
    Let~$A$ be the first clump in~$P_1$. 
    Let~$A'$ be the parts of~$\PP$ that intersect~$A$. 
    Since~${G[A]}$ is connected, ${H[A']}$ is connected. 
    Since ${\abs{A'} \leq \abs{A} = t-1}$, there are at most~${(t-1)\ell}$ vertices in the parts in~$A'$. 
    Thus at most~${(t-1)\ell}$ clumps in~$P_1$ intersect parts in~$A'$. 
    Hence~$P_1$ contains a clump~$B$ that intersects no part in~$A'$. 
    Let~$B'$ be the parts in~$\PP$ that intersect~$B$. 
    Thus~${A' \cap B' = \emptyset}$. 
    Since~${G[B]}$ is connected, ${H[B']}$ is connected. 
    Since~$P_1$ is connected and no vertex of~$P_1$ is in a part in~$C'$, there is a path~$Q$ in~${H-C'}$ from~$A'$ to~$B'$, with no internal vertex of~$Q$ in~${A'\cup B'}$. 
    Let~$H'$ be obtained from~$H$ by contracting~${H[A']}$ to a vertex~$a$, contracting~${H[B']}$ to a vertex~$b$, and contracting~$Q$ down to an edge~$ab$. 
    Each vertex in~$C'$ has a neighbour in~$A$ and in~$B$. 
    Thus~${C'\cup\{a,b\}}$ is a ${(t+1)}$-clique in~$H'$. 
    Thus~$H$ contains~$K_{t+1}$ as a minor, contradicting that~${\tw(H) \leq t-1}$. 
    Therefore the underlying treewidth of the class of $K_t$-topological-minor-free graphs equals~$t$.
\end{proof}

The proof of \cref{NoKstMinorUnderlyingTreewidth} is easily adapted to show that for~${s \leq 3}$ and fixed~$t$, every $K_{s,t}$-topological-minor-free graph~$G$ of treewidth~$k$ has $s$-tree-partition-width~${O(k^2)}$; see \cref{K3tTopoMinorPartition}. 
So the underlying treewidth of the class of $K_{s,t}$-topological minor-free graphs equals~$s$ if~${s \leq 3}$ and~${t \geq \max\{s,3\}}$. 
But the proof does not generalise for~${s \geq 4}$. 
Determining the underlying treewidth of the class of $K_{s,t}$-topological-minor-free graphs is an interesting open problem. 
Also, in the~${s = 2}$ case we can use Menger's Theorem instead of \cref{EP-treewidth} to show that every $K_{2,t}$-topological-minor-free graph of treewidth~$k$ has 2-tree-partition-width~${O(t^2k)}$; see \cref{K2tTopoMinorPartition}. 
A result of \citet[(4.4)]{LeafSeymour15} implies that every $K_{2,t}$-minor-free graph has treewidth at most~$3t/2$. 
Thus every $K_{2,t}$-minor-free graph has 2-tree-partition-width~${O(t^3)}$. 
It is open whether every $K_{2,t}$-minor-free graph has 2-tree-partition-width~${O(t)}$, which would be best possible by \cref{cTreePartitionWidthTreewidth}. 
It is easily seen that every $K_{1,t}$-minor-free graph has tree-partition-width~${O(t)}$; see \cref{K1tMinorPartition}. 


We now show that $c$-tree-partition-width is well-behaved under subdivisions (for~${c \geq 1}$).

\begin{lem}
    \label{monotonesubdivision} 
    For~${c,t \in \NN}$, if~$\tilde{G}$ is a subdivision of a graph~$G$ with~${\tpw_c(G)\leq t}$, then ${\tpw_c(\tilde{G}) \leq t^2+t}$ if~${c = 1}$ and~${\tpw_c(\tilde{G}) \leq t}$ if~${c \geq 2}$.
\end{lem}

\begin{proof}
    We first prove the~${c = 1}$ case. 
    If~${t = 1}$, then~$G$ is a forest. 
    So~$\tilde{G}$ is a forest and the claim follows. 
    Now assume~${t \geq 2}$. 
    Let ${(V_x \colon x \in V(T))}$ be a tree-partition of~$G$ with width at most~$t$. 
    We construct a tree~$T'$ from~$T$ and a $T'$-partition of~$\tilde{G}$ iteratively as follows. 
    Orient the edges of~$T$ so that each vertex has in-degree at most~$1$. 
    Say an edge~$vw$ of~$G$ is replaced by the ${(v,w)}$-path~${v_0,v_1,\dots,v_s,v_{s+1}}$ in~$\tilde{G}$. 
    If~$v$ and~$w$ are in the same part~$V_x$ (where~${x \in V(H)}$), then add a path of bags ${V_x,\{v_1,v_{s}\},\{v_2,v_{s-1}\},\ldots}$ to~$T'$. 
    If~${v \in V_x}$ and~${w \in V_y}$ then~${xy \in E(T)}$. 
    Assume~$xy$ is directed from~$x$ to~$y$. 
    Add~$v_1$ to~$V_y$ and add a path of bags ${V_y,\{v_2,v_{s}\},\{v_3,v_{s-2}\},\ldots}$ to~$T'$. 
    Since~${\abs{\{uw\in E(G) \colon u \in V_x, w \in V_y\}} \leq t^2}$, we have~${\abs{V_x}\leq t^2+t}$ for all~${x \in V(T')}$. 
    This defines a tree-partition of~$\tilde{G}$ with width at most~${t^2+t}$. 
    
    Now consider the~${c \geq 2}$ case. 
    If~${t = 1}$, then~${\tw(G) \leq c}$. 
    Since subdividing edges does not increase treewidth, ${\tw(\tilde{G}) \leq c}$, and the claim follows. 
    Now assume~${t \geq 2}$. 
    Let~${(V_x \colon x\in V(H))}$ be an $H$-partition of~$G$ with width at most~$t$ where~${\tw(H) \leq c}$. 
    We construct a graph~$H'$ of treewidth at most~$c$ from~$H$ and a $H'$-partition of~$\tilde{G}$ iteratively as follows. 
    Say an edge~$vw$ of~$G$ is replaced by the $(v,w)$-path~${v_0,v_1,\dots,v_s,v_{s+1}}$ in~$\tilde{G}$. 
    If~$v$ and~$w$ are in the same part~$V_x$ (where~${x \in V(H)}$), then add a path of bags ${V_x,\{v_1,v_{s}\},\{v_2,v_{s-1}\},\dotsc}$ to~$H'$. 
    If~${v \in V_x}$ and~${w \in V_y}$ then~${xy \in E(H)}$; add a path of bags ${\{v_1,v_{s}\},\{v_2,v_{s-1}\},\dotsc}$ to~$H'$, where~${\{v_1,v_{s}\}}$ is adjacent to~$x$ and~$y$ in~$H'$. 
    This defines a $c$-tree-partition of~$\tilde{G}$ with width at most~$t$, as desired. 
\end{proof}

\begin{lem}
    \label{TopologicalMinor}
    For all~${c \in \NN}$, for every graph~$G$ and every subdivision~$G'$ of~$G$, 
    \begin{equation*}
        \tpw_c(G)\leq 4c^2 12^c (c \tpw_c(G')+1) \tpw_c(G')^{2} (\tw(G')+1)^c.
    \end{equation*}
\end{lem}

\begin{proof}
    Let~${\ell \coloneqq \tpw_c(G')}$. 
    By \cref{disjointed}, $G'$ has a ${(c,c\ell)}$-disjointed $\ell$-partition~$\beta'$. 
    Let ${\beta \coloneqq \{B' \cap V(G) \colon B' \in \beta'\}}$ be the corresponding $\ell$-partition of~$G$. 
    We show that~$\beta$ is ${(c,2c\ell(c\ell+1))}$-disjointed. 
    The result then follows from \cref{MainCorollary} as~${\tw(G') = \tw(G)}$. 
    
    Let~${B_1, \dotsc, B_c \in \beta}$ be distinct non-empty parts and let ${B'_1, \dotsc, B'_c}$ be the corresponding parts in~$\beta'$. 
    Let~${B \coloneqq \bigcup \{B_{i} \colon i \in \{1,\dots,c\}\}}$  and~${B' \coloneqq \bigcup \{B'_{i} \colon i \in \{1,\dots,c\}\}}$. 
    Observe that every vertex in $B' - B$ is an internal vertex of some subdivided edge in~$G$. 
    Let~${\tilde{Q} \subseteq V(G)}$ be the set of end-vertices of these subdivided edges, so~${\abs{\tilde{Q}} \leq 2 \abs{B' - B} < 2c \ell}$. 
    
    Let~$X$ be a component of~${G - B}$ and let~${Y'_1, \dotsc, Y'_m}$ be the components of~${G' - B'}$ that intersect~$X$. 
    Since~$\beta'$ is ${(c, c\ell)}$-disjointed, for each~$j$ there is some~${Q'_j \subseteq V(Y'_j)}$ of size at most~$c\ell$ such that every component of~${Y'_j - Q'_j}$ is disjoint from some~${N_{G'}(B'_i)}$. 
    Define~$Q_j$ as follows. 
    For each~${w \in Q'_j}$ either~${w \in V(G)}$ or~$w$ is an internal vertex of some subdivided edge~$uv$ in~$G$. 
    In the first case add~$w$ to~$Q_j$, and in the second case add~$u$ and~$v$ to~$Q_j$. 
    Thus~${\abs{Q_j}\leq 2c\ell}$. 
    
    Let~${Q \coloneqq (\bigcup \{Q_j \colon j \in \{1,\dots,m\} \} \cup \tilde{Q}) \cap V(X)}$. 
    Then~${\abs{Q} \leq 2m c \ell + 2c \ell = 2c \ell (m + 1)}$. 
    Let~$Z$ be a component of~${X - Q}$. 
    Consider any path~$P$ within~$Z$: 
    in~$G'$ this is subdivided to give a path~$P'$. 
    As~$P$ avoids~$\tilde{Q}$, the path~$P'$ avoids~$B'$. 
    Hence~${V(Z)}$ is contained in some component of~${G' - B'}$ and so within some~$Y'_j$. 
    Further, $P$ avoids~$Q_j$, so~$P'$ avoids~$Q'_j$. 
    Hence~${V(Z)}$ is contained in some component~$Z'$ of~${Y'_j - Q'_j}$. 
    By the definition of~$Q'_j$, there is some~$i$ with~${N_{G'}(B'_i) \cap V(Z') = \emptyset}$.
    
    We claim that~${N_{G}(B_i) \cap V(Z) = \emptyset}$. 
    Indeed suppose~${u \in B_i}$ and~${v \in V(Z)}$ are adjacent and consider the corresponding ${(u,v)}$-path~$P'$ in~$G'$. 
    Since~${v \not \in \tilde{Q}}$, the path~${P' - u}$ avoids~$B'$ and so is within~$Y'_j$ (as~${v \in V(Z) \subseteq V(Y_j)}$). 
    As~${v \not \in Q_j}$, the path~${P' - u}$ avoids~$Q'_j$. 
    Hence~${P' - u}$ is a connected subgraph of~${Y_j - Q'_j}$ so is within~$Z'$ (as~${v \in V(Z) \subseteq V(Z')}$). 
    But then~${N_{G'}(u) \cap V(Z') \neq \emptyset}$, a contradiction. 
    Hence~$\beta$ is ${(c, 2c \ell (m + 1))}$-disjointed. 
    We are left to show that~${m \leq c \ell}$. 
    
    For each~${i \in \{1,\dots,m\}}$, let~${v_i \in V(Y'_i) \cap V(X)}$. 
    For distinct~${i,j \in \{1,\dots,m\}}$, let~$P_{i,j}$ be a ${(v_i,v_j)}$-path in~$X$. 
    Since~$Y'_i$ and~$Y'_j$ are distinct components in~${G' - B'}$, there exist a vertex~${w \in B' - B}$ that is an internal vertex of a subdivided edge in~$P_{i,j}$. 
    Therefore, ${m - 1 \leq \abs{B' - B} < c \ell}$. 
\end{proof}

\begin{thm}
    \label{MonotoneUnderlyingTreewidth}
    A monotone graph class~$\GG$ has bounded underlying treewidth if and only if~$\GG$ excludes some fixed topological minor.
\end{thm}

\begin{proof}
    ($\Longleftarrow$) Say~$\GG$ excludes some fixed topological minor~$X$ on~$p$ vertices. By \cref{TopoMinorPartition}, every graph~${G \in \GG}$ has $p$-tree-partition-width at most~${O(\tw(G)^2)}$. 
    Thus~$\GG$ has underlying treewidth at most~$p$. 
    
    ($\Longrightarrow$) Say~$\GG$ has underlying treewidth~$c$. 
    That is, there is a function~$f$ such that ${\tpw_c(G) \leq f(\tw(G))}$ for every graph~${G \in \GG}$. 
    For any~${n \in \NN}$, by \cref{StandardExamples}, there is a graph~$H$ of treewidth~${c + 1}$ and~${\tpw_c(H) > n}$. 
    Suppose for the sake of contradiction that some subdivision~$H'$ of~$H$ is in~$\GG$. 
    Thus~${\tpw_c(H')\leq f( \tw(H')) = f(\tw(H)) = f(c+1)}$. 
    By \cref{TopologicalMinor},
    \begin{align*}
        n \;<\; \tpw_c(H) 
        \;\leq\; & 4c^2 12^c\,(c \tpw_c(H') + 1)\tpw_c(H')^{2}\, (\tw(H') + 1)^c \\
        \;\leq\; & 4c^2 12^c\,(cf(c+1) + 1) (f(c + 1))^{2}\, (c+2)^c.
    \end{align*}
    We obtain a contradiction taking~${n \gg f(c+1)}$. 
    Thus no subdivision of~$H$ is in~$\GG$. 
    Since~$\GG$ is monotone, $\GG$ excludes~$H$ as a topological minor, as desired. 
\end{proof}

\section{Excluding a Subgraph}
\label{sec:subgraph}

For a graph $H$, a graph $G$ is \defn{$H$-free} if $G$ contains no subgraph isomorphic to $H$. For a finite set of graphs $\mathcal{H}$, we say that $G$ is \defn{$\mathcal{H}$-free} if $G$ is $H$-free for all $H \in \mathcal{H}$. Let \defn{$\GG_H$} be the class of $H$-free graphs and let \defn{$\GG_{\mathcal{H}}$} be the class of $\mathcal{H}$-free graphs. This section characterises when $\GG_{\mathcal{H}}$ has bounded underlying treewidth, and determines the exact underlying treewidth for  several natural classes. A \defn{spider} is a subdivision of a star and a \defn{spider-forest} is a subdivision of a star-forest. For $s,t\in\NN$ with $s\geq 2$, the \defn{$(s,t)$-spider}, denoted \defn{$S_{s,t}$}, is the $(t-1)$-subdivision of $K_{1,s}$. If $v$ is the centre of $S_{s,t}$, then each component of $S_{s,t}-v$ is called a \defn{leg}.

\begin{thm}
\label{SummaryTheorem}
For all $\ell,n,s,t\in \NN$ where $n,s\geq 3$ and $\ell \geq 2$, and for every finite set $\HH$ of graphs,
\begin{enumerate}[label = \textnormal{(\roman*)}]
\item $\GG_{\HH}$ has bounded underlying treewidth if and only if $\HH$ contains a spider-forest\textnormal{;}
\item the underlying treewidth of $\GG_{P_n}$ equals $\floor{\log n} - 1$\textnormal{;}
\item  the underlying treewidth of $\GG_{\ell\, P_n}$ equals $\floor{\log n}$\textnormal{;}
\item the underlying treewidth of $\GG_{S_{s,t}}$ equals $\floor{\log t} + 1$\textnormal{;}
\item the underlying treewidth of $\GG_{\ell\, S_{s,t}}$ equals $\floor{\log t} + 2$.
\end{enumerate}
\end{thm}

We prove \cref{SummaryTheorem} through a sequence of lemmas.

\begin{lem}
\label{ExcludingSubgraph}
For every finite set $\mathcal{H}$ of graphs, if $\GG_{\mathcal{H}}$ has bounded underlying treewidth, then $\mathcal{H}$ contains a spider-forest.
\end{lem}

\begin{proof}
For the sake of contradiction, suppose that $\mathcal{H}$ does not contain a spider-forest. Then each graph $H \in \mathcal{H}$ contains a cycle or two vertices with degree at least $3$ in the same component of $H$. Let $s_H$ be the minimum of the girth of $H$ and the minimum distance between distinct vertices with degree at least $3$ in the same component of $H$. Let $s \coloneqq\max\{s_H\colon H\in \HH\}$ and let $\GG$ be any graph class with unbounded underlying treewidth (such as that in \cref{StandardExampleUnderlyingTreewidth}). Let $\GG'$ be the class of $(s+1)$-subdivisions of graphs in $\GG$. Then $\GG'$ is $H$-free for all $H \in \HH$ and has unbounded underlying treewidth by \cref{TopologicalMinor}, a contradiction.
\end{proof}
\cref{ExcludingSubgraph} proves the necessity of (i) in \cref{SummaryTheorem}.

We now work towards showing $\GG_H$ has bounded underlying treewidth when $H$ is a spider. For a graph $G$ and $\lambda\in \NN$, let \defn{$G^{(\lambda)}$} be the graph with $V(G^{(\lambda)})=V(G)$ and $uv\in E(G^{(\lambda)})$ whenever there are $\lambda$ internally disjoint $(u,v)$-paths in $G$. 

\begin{lem}\label{lambdatw}
For all $k,\lambda \in \NN$ where $\lambda\geq k+1$, if a graph $G$ has treewidth at most $k$, then  $\tw(G^{(\lambda)})\leq k$.
\end{lem}

 \begin{proof}
    Let $\WW\coloneqq (W_x \colon x\in V(T))$ be a tree-decomposition of $G$ with width $k$. We may assume that $W_x\neq W_y$ for each $xy\in E(T)$. We claim that $\WW$ is also a tree-decomposition of $G^{(\lambda)}$, for which it suffices to show that for every edge $vw$ of $G^{(\lambda)}$, $v$ and $w$ are in a common bag of $\WW$. To this end, suppose that $vw\in E(G^{(\lambda)})$ but there is no bag of $\WW$ containing both $v$ and $w$. Then there is an edge $xy\in E(T)$ separating $\{z\in V(T) \colon v\in W_z\}$ and $\{z\in V(T) \colon w\in W_z\}$, which implies that $W_x\cap W_y$ is a set of at most $k$ vertices separating $v$ and $w$ in $G$. On the other hand, since $vw\in E(G^{(\lambda)})$ and $\lambda\geq k+1$, any vertex-set separating $v$ and $w$ must have at least $k+1$ vertices, a contradiction.
\end{proof}



\begin{lem}\label{Jlambda2}
For all $s,t \in \NN$ and $\lambda\geq 1+s+st(2t+1)$, if a graph $G$ contains no $S_{s,2t+1}$ subgraph, then $G^{(\lambda)}$ contains no $S_{s,t}$ subgraph.
\end{lem}

\begin{proof}
Suppose for contradiction that $S_{s,t}$ is a subgraph of $G^{(\lambda)}$. Since $G$ does not contain $S_{s,2t+1}$ as a subgraph, there is no set of $s$ internally disjoint paths of length at least $2t+2$ in $G$ between any pair of vertices. Hence, since $\lambda\geq 1+s+st(2t+1)$ we may greedily replace each edge $uv$ of $S_{s,t}$ with a $(u,v)$-path in $G$ of length between $2$ and $2t+1$ such that all these paths are internally vertex disjoint from each other, and each of these paths internally avoids $V(S_{s,t})$. This works because the number of vertices to avoid when replacing an edge $uv\in E(S_{s,t})$  is at most $1+(2t+1)(\abs{E(S_{s,t})\setminus \{uv\}})=1+(2t+1)(st-1)$, and there is a collection of $\lambda-s\geq 1+st(2t+1)$ internally disjoint $(u,v)$-paths in $G$ of length between $2$ and $2t+1$. Finally, since each leaf of $S_{s,t}$ has degree at least $\lambda\geq 1+s+st(2t+1)$ in $G$, we can extend each leg of the spider thus constructed by one further vertex by adding a distinct neighbour in $G$ of each leaf of $S_{s,t}$. We obtain $S_{s,2t+1}$ as a subgraph of $G$, a contradiction.
\end{proof}

\begin{lem}\label{lambdatpw}
There exists a function $f$ such that for all $c,k,s,t,\lambda,\gamma\in \NN$, if a graph $G$ has treewidth at most $k$ and no $S_{s,t}$ subgraph, and $\tpw_c(G^{(\lambda)}) \leq \gamma$, then
$\tpw_{c+1}(G) \leq f(c,k,s,t,\lambda,\gamma)$.
\end{lem}

\begin{proof}
Let $d \coloneqq st$, $\Delta \coloneqq s\cdot (1+\lambda^{2t})$, $m \coloneqq (1+\Delta^{d})\cdot \gamma^{c+1}$ and $n \coloneqq (k+1)m$. Let $\beta \coloneqq (V_h \colon h\in V(H))$ be an $H$-partition of $G^{(\lambda)}$ with width at most $\gamma$, where $\tw(H) \leq c$. We show that $\beta$ is a $(c+1, n)$-disjointed $\gamma$-covering of $G$. By \cref{MainCorollary}, $G$ has $(c+1)$-tree-partition-width at most
    $f(c,k,s,t,\lambda,\gamma)\coloneqq 2 (c+1)\cdot  n\cdot \gamma (12 k)^c.$

    Consider distinct $V_1,\dotsc, V_{c+1} \in \beta$ and a connected component $X$ of $G -(V_1 \cup\dotsb \cup V_{c+1})$.
    Let $\mathcal{F}$ be the collection of connected subgraphs in $X$ that intersect $N_G(V_i)$ for each $i\in \{1,\dotsc,c+1\}$. Let $\mathcal{R}'$ be a maximum-sized set of vertex-disjoint elements of $\mathcal{F}$. By \cref{EP-treewidth}, it suffices to show that $\abs{\mathcal{R}'}< m$. 
    
    We make two simplifying assumptions. 
    First, we may assume that each $V_{i}$ consists of a single vertex at the cost of a factor of $\gamma^{c+1}$:
    by averaging over all $(c+1)$-tuples in $V_1\times \dotsb\times V_{c+1}$, there is one such tuple $(x_1, \dotsc, x_{c+1})$ and a subset $\mathcal{R}\subset \mathcal{R}'$ such that $\abs{\mathcal{R}'}\leq \gamma^{c+1}\abs{\mathcal{R}}$ and $N_{G}(x_i) \cap V(R) \neq \emptyset$ for each $R\in \mathcal{R}$ and $i\in \{1,\dotsc, c+1\}$. Thus it suffices to show that $\abs{\mathcal{R}}\leq 1+\Delta^{d}$.
    Second, we may assume that $\bigcup\{V(R) \colon R\in\mathcal{R}\}=V(X)$.
    
    Let $X'$ be the quotient graph of $\mathcal{R}$ (with respect to $X$). If $\abs{V(X')} < \lambda$, then we are done. Thus we may assume that $\abs{V(X')}  \geq \lambda$ and so the vertices $\{x_1,\dotsc, x_{c+1}\}$ form a $(c+1)$-clique in $G^{(\lambda)}$. To show that $\abs{\mathcal{R}}=\abs{V(X')}\leq 1+\Delta^{d}$, we use the well-known fact that every graph with diameter less than $d$ and maximum degree $\Delta$ has at most $1+\Delta^{d}$ vertices \cite{MS-EJC05}.
    
    \begin{claim}
    $X'$ contains no path on $d$ vertices and thus has diameter less than $d$.
    \end{claim}
    
    \begin{proof}
    For the sake of contradiction, suppose that $X'$ contains a path $P'=(R_0,\dotsc ,R_{d-1})$ on $d$ vertices. 
    Consider the vertex-disjoint paths $(P_j' \colon j\in \{1,\dotsc, s\})$ in $X'$ where $P_j'\coloneqq  (R_{(j-1)t}, \dotsc, R_{jt})$. 
    For each $j\in \{1,\dotsc,s\}$ fix $u_{j}\in V(R_{(j-1)t})\cap N_G(x_1)$. Then there is a $(u_j,v_j)$-path $P_j$ in $X[\bigcup(V(R) \colon R\in V(P_j'))]$ of length at least $t-1$ for some vertex $v_j\in V(R_{jt})$. Since the paths $(P_j\colon j\in \{1,\dotsc, s\})$ are pairwise vertex-disjoint, it follows that $G[\{x_1\} \cup \bigcup(V(P_j) \colon j \in \{1,\dots,s\})]$ contains $S_{s,t}$ (see \cref{fig:a}), a contradiction.
    \end{proof}
    
    It remains to show that $X'$ has maximum degree less than $\Delta$. For the sake of contradiction, suppose that $X'$ contains a vertex $R'$ with degree at least $\Delta$. 
    Let $R_1,\dotsc,R_{\Delta}$ be  neighbours of $R'$ and for each $i \in \{1,\dotsc,\Delta\}$, let $y_i\in V(R_i)$ be adjacent to some $w_i\in V(R')$. Let $L\coloneqq\{y_i \colon i\in \{1,\dotsc,\Delta\}\}$ and let $R^*$ be obtained from $R'$ by adding $L$ and $\{y_{i}w_{i} \colon i\in \{1,\dotsc,\Delta\}\}$. Let $T$ be a vertex-minimal tree in $R^*$ such that $L\subseteq V(T)$. Then $L$ is precisely the set of leaves of $T$.

    \begin{claim}
    $T$ has maximum degree less than $\lambda$.
    \end{claim}

    \begin{proof}
    For the sake of contradiction, suppose there is a vertex $v\in V(T)$ with degree (in $T$) at least $\lambda$. Then there exists a set of $\lambda$ internally vertex-disjoint paths from $v$ to distinct leaves $y_1,\dotsc, y_\lambda$ in $T$. Then for each $i\in \{1,\dotsc, c+1\}$ and $j\in \{1, \dotsc, \lambda\}$, since $R_j$ is connected, there is an $(x_i,y_j)$-path in $G[V(R_j)\cup \{x_i\}]$. As such, there exists $\lambda$ internally disjoint $(v,x_i)$-paths in $G$ for each $i\in \{1,\dots,c+1\}$ (see \cref{fig:b}). Thus $vx_i\in E(G^{(\lambda)})$. So $\{x_1,\dots,x_{c+1},v\}$ is a $(c+2)$-clique in $G^{(\lambda)}$. However, as each vertex of this clique belongs to a different part in $\beta$, it follows that $H$ contains a $(c+2)$-clique, which contradicts $\tw(H)\leq c$.
    \end{proof}
    Thus, balls of radius $r$ (in $T$) have fewer than $1+\lambda^{r}$ elements. Since $\abs{L}=\Delta= s\cdot(1+\lambda^{2t})$, there are vertices $y_1,\dotsc, y_s\in L$ whose balls of radius $t$ in $T$ are pairwise disjoint. Let $u_i\in V(R_i)\cap N_G(x_1)$ and let $v_i\in V(T)$ be a vertex with distance (in $T$) exactly $t$ from $y_i$. By construction, there exists a set of disjoint paths $(P_i \colon i\in \{1,\dots,s\})$ such that $P_i$ is a $(u_i,v_i)$-path of length at least $t$. Then, $G[\{x_1\}\cup \bigcup(V(P_i) \colon i\in \{1,\dots,s\})]$ contains $S_{s,t}$ (see \cref{fig:c}), a contradiction. Thus $X'$ has maximum degree less than $\Delta$.
\end{proof}

\begin{figure} 
    \captionsetup[subfigure]{labelformat=empty}
	\centering
	\subcaptionbox{$\diam(X')\geq d$\label{fig:a}}{
	\includegraphics{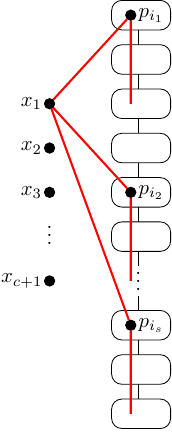}
	}\hspace{3.5em}
	\subcaptionbox{$\Delta(X')\geq \Delta$, $\Delta(T)\geq \lambda$ \label{fig:b}}{
	\includegraphics{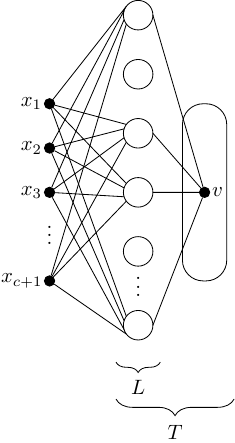}
	}\hspace{3.5em}
	\subcaptionbox{$\Delta(X')\geq \Delta$, $\Delta(T)< \lambda$ \label{fig:c}}{
	\includegraphics{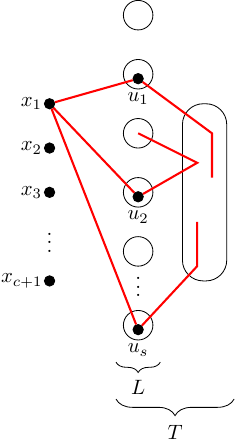}
	}
    \caption{Finding spiders and cliques in the proof of~\cref{lambdatpw}.}
\end{figure}





\begin{lem}\label{ExcludingPath}
For each $n \in \NN$ with $n\geq 3$, the underlying treewidth of $\GG_{P_n}$ equals $\floor{\log{n}} - 1$.
\end{lem}

\begin{proof}
Let $c\coloneqq \floor{\log{n}}$. For the lower bound, \cref{StandardExamples} (ii) and (vii) imply that $\{C_{c-1,\ell}\colon \ell\in \NN\}$ has underlying treewidth at least $c-1$ and is a subset of  $\GG_{P_{n}}$. The lower bound follows.  

We prove the upper bound by induction on $c\geq 1$ with the hypothesis that there is a non-decreasing function $f$ such that every $P_{2^{c+1}-1}$-free graph $G$ has $(c-1)$-tree-partition-width at most $f(\tw(G),c)$. Since $P_n\subseteq P_{2^{c+1}-1}$, the claim follows. We make no attempt to optimise $f$. When $c=1$, $G$ contains no $P_3$ and so each component has at most two vertices. Therefore, $\tpw_0(G)\leq f(\tw(G),1) \coloneqq 2$.

Now assume $c\geq 2$ and the lemma holds for $c-1$. Let $G$ be a graph with no $P_{2^{c+1}-1}$ and let $\lambda \coloneqq \max\{3+(2^{c}-2)(2^{c}-1), \tw(G)+1\}$. By \cref{Jlambda2,lambdatw}, $G^{(\lambda)}$ has treewidth at most $\tw(G)$ and contains no $P_{2^{c}-1}$. By induction, $\tpw_{c-2}(G^{(\lambda)}) \leq \gamma\coloneqq f(\tw(G),c-1)$. By \cref{lambdatpw}, $\tpw_{c+1}(G) \leq f(\tw(G),c)\coloneqq \tilde{f}(c,\tw(G),2,2^{c}-1,\lambda,\gamma)$ where $\tilde{f}$ is from  \cref{lambdatpw}, as required. 
\end{proof}
\cref{ExcludingPath} proves (ii) in \cref{SummaryTheorem}.

To prove (iv) in \cref{SummaryTheorem}, we need the following lower bound result.

\begin{lem}
\label{LowerBoundSpider}
For all $c\in\NN$, there exists an $S_{3,2^c}$-free graph class $\{J_{c,N} \colon N\in \NN\}$ with underlying treewidth $c+1$.
\end{lem}

\begin{proof}
Let $J_{c,N}$ be the graph obtained from a path $P_{N+1}=(p_1,\ldots, p_{N+1})$ where for each edge $p_{i}p_{i+1}\in E(P_{n+1})$, we add a copy $X_{i}$ of $(2N-1)C_{c-1,N}$ and let $p_i$ and $p_{i+1}$ dominate $X_{i}$ (see~\cref{SpiderLB}; this construction also provides a lower bound on the clustered chromatic number~\citep{NSSW19}). To prove the lemma, it suffices to show that $J_{c,N}$ is $S_{3,2^c}$-free, $\tw(J_{c,N})=c+1$ and $\tpw_c(J_{c,N})>N$.

The claimed treewidth of $J_{c,N}$ holds since $\tw((2N-1)C_{c-1,N})=c-1$ (\cref{StandardExamples} (i)), and the two dominant vertices for each copy of $(2N-1)C_{c-1,N}$ give $\tw(J_{c,N})=c+1$. Now suppose, for the sake of contradiction, that $S_{3,2^c}$ is a subgraph of $J_{c,N}$ with centre $v\in V(J_{c,N})$. Observe that $v$ has at most two neighbours $p_i$ and $p_j$ in $P_N$. Moreover, each neighbour of $v$ in $J_{c,N}\setminus \{v,p_i,p_j\}$ is contained in a component that is a subgraph of $C_{c-1,N}$. Since $S_{3,2^c}$ has three legs, one of the legs must be contained in a copy of $C_{c-1,N}$ which contradicts $C_{c-1,N}$ not having a path on $2^{c}$ vertices (\cref{StandardExamples} (vii)).

It remains to show that $\tpw_c(J_{c,N})> N$. For the sake of contradiction, suppose that $J_{c,N}$ has a $c$-tree-partition $\beta\coloneqq(V_h \colon h\in V(H))$ with width at most $N$. Since $\abs{V(P_{N+1})}=N+1$, there exists an edge $p_ip_{i+1}\in E(P_{N+1})$ such that $p_i$ and $p_{i+1}$ belong to different parts in $\beta$. Let $V_i$ and $V_{i+1}$ be such that $p_i\in V_i$ and $p_{i+1}\in V_{i+1}$. Since $\abs{(V_i\cup V_{i+1})\setminus\{p_i,p_{i-1}\}}\leq 2N-2$, there exists a copy of $C_{c-1,N}$ in $X_i$ such that $V(C_{c-1,N})\cap (V_i\cup V_{i+1})=\emptyset$. By \cref{StandardExamples} (ii), $C_{c-1,N}$  contains a $c$-clique whose vertices are in different parts in $\beta$. Together with $p_i$ and $p_{i+1}$, it follows that $J_{c,N}$ contains a $(c+2)$-clique whose vertices are in different parts in $\beta$. Thus $H$ has a $(c+2)$-clique, a contradiction.
\end{proof}

\begin{figure}[ht]
    \centering
    \includegraphics{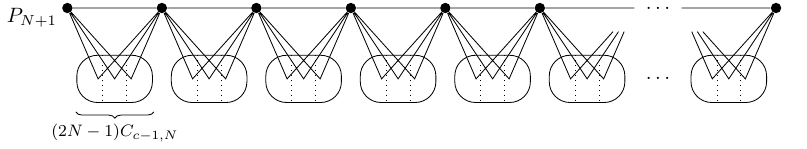}
    \caption{Construction of $J_{c,N}$.}
    \label{SpiderLB}
\end{figure}

\begin{lem}
\label{stspider}
 For all $s,t \in \NN$ where $s\geq 3$, the underlying treewidth of $\GG_{S_{s,t}}$ is $\floor{\log{t}} + 1$.
\end{lem}

\begin{proof}
Let $c \coloneqq \floor{\log{t}}$. For the lower bound, $\GG_{S_{3,2^{c}}}$ has underlying treewidth at least $c+1$ by \cref{LowerBoundSpider}. Since $S_{3,2^{c}}\subseteq S_{s,t}$, it follows that $\GG_{S_{3,2^{c}}}\subseteq \GG_{S_{s,t}}$, which implies the lower bound.

We prove the upper bound by induction on $c\in\NN_0$ with the following hypothesis: There is an increasing function $f$ such that for each $s\in\NN$, every $S_{s,2^{c+1}-1}$-free graph $G$ has $(c+1)$-tree-partition-width at most $f(\tw(G),c,s)$. Since $S_{s,t}\subseteq S_{s,2^{c+1}-1}$, the claim follows. When $c=0$, $G$ has maximum degree at most $s-1$, and $G$ has $1$-tree-partition-width at most $f(\tw(G),0,s) \coloneqq 24(\tw(G)+1)(s-1)$ by \cref{TreePartitionWidthDegree}.

Now assume $c\geq 1$ and the lemma holds for $c-1$. Let $G$ be an $(s,2^{c+1}-1)$-spider-free graph and let $\lambda \coloneqq \max\{1+s+st(2t+1), \tw(G)+1\}$. By \cref{Jlambda2,lambdatw},
$\tw(G^{(\lambda)}) \leq \tw(G)$ and $G^{(\lambda)}$ contains no $S_{s,2^{c}-1}$. By induction, $\tpw_c(G^{(\lambda)}) \leq \gamma\coloneqq f(\tw(G),c, s)$. By \cref{lambdatpw}, $\tpw_{c+1}(G) \leq f(\tw(G),c, s)\coloneqq \tilde{f}(c,\tw(G),s,2^{c+1}-1,\lambda,\gamma)$ where $\tilde{f}$ is from \cref{lambdatpw}, as required. 
\end{proof}

\cref{stspider} proves (iv) in \cref{SummaryTheorem}.

\begin{lem}
\label{ExcludingSpiderForest}
For every connected graph $H$ and $\ell \in \NN$ where $\ell\geq 2$, if $\GG_H$ has underlying treewidth $c$ then $\GG_{\ell \, H}$ has underlying treewidth $c+1$. 
\end{lem}

\begin{proof}
We first prove the lower bound. Since $\GG_H$ has underlying treewidth $c$, there exists $k\in \NN$ such that for all $j\in \NN$, there is a graph $G_j\in \GG_H$ with $\tw(G_j)=k$ and $\tpw_{c-1}(G_j)>j$. Consider the graph class $\mathcal{J}\coloneqq\{\widehat{j\,G_j} \colon j\in \NN\}.$ Observe that every graph in $\mathcal{J}$ is $\ell \,H$-free, so $\mathcal{J}\subseteq \GG_{\ell \, H}$. Moreover, since adding a dominant vertex increase the treewidth of a graph by $1$, $\tw(\widehat{j\,G_j})=k+1$ for all $j\in \NN$. Finally, consider a $j$-partition $\beta\coloneqq(V_x\colon x \in V(J))$ of $\widehat{j\, G_j}$. At most $j - 1$ copies of $G_j$ contain a vertex in the same part $V_y$ as the dominant vertex $v$. Thus, some copy $G$ of $G_j$ contains no vertex in $V_y$. Consider the sub-partition $\beta'\coloneqq(V_x\cap V(G)\colon x \in V(J'))$ of $G$ where $J'\subseteq J$. Since $\beta'$ has width at most $j$, $J'$ has treewidth at least $c$ and so $J[V(J')\cup \{y\}]$ has treewidth at least $c+1$. Thus $\mathcal{J}$ has underlying treewidth at least $c+1$. 

We now prove the upper bound. Let $G$ be an $\ell \, H$-free graph. Let $A_1,\dots,A_m$ be a maximal set of pairwise disjoint copies of $H$ in $G$. Then $m< \ell$. Let $B \coloneqq G-V(A_1\cup\dots\cup A_m)$. By the maximality of $m$, $B$ is $H$-free. Thus $B$ has a $c$-tree-partition $(V_x \colon x\in V(J'))$ with width at most $\tilde{f}(\tw(B))$ for some function $\tilde{f}$. Let $J$ be the graph obtained from $J'$ by adding a dominant vertex $y$, and let $V_y \coloneqq V(A_1\cup\dots\cup A_n)$. Then $(V_x \colon x\in V(J))$ is a $(c+1)$-tree-partition of $G$ with width at most $f(\tw(G))\coloneqq \max\{(\ell-1)\abs{V(H)},\tilde{f}(1),
\dots,\tilde{f}(\tw(G))\}$, as required.
\end{proof}

\cref{ExcludingPath,stspider,ExcludingSpiderForest} prove \cref{SummaryTheorem} (iii) and (iv). Since every spider-forest is a subgraph of $\ell \, S_{s,t}$ for some  $\ell,s,t \in \NN$, these results also imply the sufficiency of (i) in \cref{SummaryTheorem}.

\section{Excluding an Induced Subgraph}
\label{InducedSubgraph}

For a graph $H$, let \defn{$\II_H$} be the class of graphs with no induced subgraph isomorphic to $H$. This section characterises the graphs $H$ such that $\II_H$ has bounded underlying treewidth, and determines the precise 
underlying treewidth for each such $H$. 

\begin{thm}
\label{InducedSubgraphUnderlyingTreewidth}
For any graph $H$, 
\begin{enumerate}[label = \textnormal{(\roman*)}]
    \item $\II_H$ has bounded underlying treewidth if and only if $H$ is a star-forest\textnormal{;}
    \item if $H$ is a star-forest, then $\II_H$ has underlying treewidth at most $2$\textnormal{;}
    \item $\II_H$ has underlying treewidth at most $1$ if and only if $H$ is a star or each component of $H$ is a path on at most three vertices\textnormal{;}
    \item $\II_H$ has underlying treewidth $0$ if and only if $H$ is a path on at most three vertices, or $E(H) = \emptyset$. 
\end{enumerate}
\end{thm}

We prove \cref{InducedSubgraphUnderlyingTreewidth} by a sequence of lemmas. 

\begin{lem}
\label{Induced-UTW-Necessary}
If $\II_H$ has bounded underlying treewidth, then $H$ is a star-forest.
\end{lem}

\begin{proof}
Suppose that $\II_H$ has underlying treewidth $c$. So $\GG_H$ has underlying treewidth at most $c$ (since $\GG_H \subseteq \II_H$). \Cref{ExcludingSubgraph} implies that $H$ is a spider-forest. By \cref{StandardExampleUnderlyingTreewidth}, $\{C_{c, \ell} \colon c, \ell \in \NN\}$ has unbounded underlying treewidth. Hence, $H$ is an induced subgraph of some $C_{c, \ell}$. By \cref{StandardExamples}, $C_{c, \ell}$ has no induced $P_{4}$ and so $H$ contains no induced $P_{4}$. Every spider-forest with no induced $P_{4}$ is a star-forest. So $H$ is a star-forest. 
\end{proof}

\cref{Induced-UTW-Necessary} proves the necessity of (i) in \cref{InducedSubgraphUnderlyingTreewidth}.

\begin{lem}
\label{ExcludeStar}
If $H$ is a star $K_{1,s}$, then $\II_H$ has underlying treewidth at most $1$. 
\end{lem}

\begin{proof}
Say $G\in\II_H$. 
For each vertex $v\in V(G)$, if $G_v \coloneqq G[N_G(v)]$ then $\alpha(G_v)\leq s-1$ and $\chi(G_v)\leq \tw(G_v)+1\leq\tw(G)$, implying $\deg_G(v)= \abs{V(G_v)} \leq\chi(G_v)\alpha(G_v)\leq \tw(G)(s-1)$. 
By \cref{TreePartitionWidthDegree}, $\tpw(G)\leq 24\tw(G)^2(s-1)$. Hence $\II_H$ has underlying treewidth at most 1. 
\end{proof}

\begin{lem}
\label{ExcludeStarForest}
If $H$ is a star-forest, then $\II_H$ has underlying treewidth at most $2$. 
\end{lem}

\begin{proof}
Say $H$ has $\ell$ components, each with at most $s$ leaves. Let $G\in \II_H$. Let $A_1,\dots,A_n$ be a maximal set of pairwise disjoint induced $K_{1,s}$ subgraphs in $G$. Define $J$ to be the graph obtained from $G[V(A_1\cup\dots\cup A_n)]$ by contracting each $A_i$ into a vertex $v_i$. Then $n\leq\chi(J)\alpha(J)\leq (\tw(J)+1)(\ell-1)\leq(\tw(G)+1)(\ell-1)$ since $J$ is a minor of $G$. Thus $\abs{V(A_1\cup \dots\cup A_n)} \leq (\tw(G)+1)(\ell-1)(s+1)$.
Let $B \coloneqq G-V(A_1\cup\dots\cup A_n)$. By the maximality of $n$, $B$ contains no induced $K_{1,s}$. As in the proof of \cref{ExcludeStar}, $B$ has a tree-partition $(V_x \colon x\in V(T))$ of width at most $24\tw(G)^2(s-1)$. Let $H$ be the graph obtained from $T$ by adding a dominant vertex 
$y$, and let $V_y \coloneqq V(A_1\cup\dots\cup A_n)$. Then $(V_h \colon h\in V(H))$ is a 2-tree-partition of $G$ with width at most $\max\{24\tw(G)^2(s-1),(\tw(G)+1)(\ell-1)(s+1)\}$.
Hence $\II_H$ has underlying treewidth at most 2. 
\end{proof}

\cref{ExcludeStarForest} proves (ii) and the sufficiency of (i) in \cref{InducedSubgraphUnderlyingTreewidth}.

\begin{lem}
\label{NoInducedP3Forest}
If $H$ is a forest in which each component is a path on at most three vertices, then $\II_H$ has underlying treewidth at most $1$. 
\end{lem}

\begin{proof}
Say $H$ has $k$ components. 
Let $G\in \II_H$.
Let $A_1,\dots,A_n$ be a maximal set of pairwise disjoint induced $P_3$ subgraphs in $G$. Let $J$ be the minor of $G$ obtained from $G[ V(A_1\cup\dots\cup A_n)]$ by contracting each $A_i$ into a vertex. Thus $n = \abs{V(J)}\leq\chi(J)\alpha(J)\leq(\tw(J)+1)(k-1)\leq(\tw(G)+1)(k-1)$. Let $B_1,\dots,B_m$ be the components of $G-V(A_1\cup\dots\cup A_n)$. By the maximality of $n$, each $B_i$ has no induced $P_3$ subgraph, so $B_i$ is a complete subgraph and $\abs{V(B_i)}\leq\tw(G)+1$. Let $T$ be the star with centre $x$ and leaves $y_1,\dots,y_m$. Let $V_x \coloneqq V(A_1\cup\dots\cup A_n)$ and $V_{y_j} \coloneqq V(B_j)$. So  $(V_h \colon h\in V(T))$ is a tree-partition of $G$ with width at most $\max\{3(k-1)(\tw(G)+1),\tw(G)+1\}$. Hence
 $\II_H$ has underlying treewidth at most 1. 
\end{proof}

\begin{lem}\label{Induced-UTW2}
If $H$ is a star-forest with at least two components and at least one component of $H$ has at least three leaves, then $\II_H$ has underlying treewidth $2$.
\end{lem}

\begin{proof}
Consider a star $S$ with at least three leaves that is a subgraph of a fan $F$ with dominant vertex $v$. Since $F-v$ has maximum degree $2$, $v$ is in $S$. Since every vertex of $F$ is adjacent to $v$, $H$ is not an induced subgraph of $F$. By \cref{StandardExamples}, the class of fan graphs has underlying treewidth 2. Thus $\II_H$ has underlying treewidth at least 2, with equality by \cref{ExcludeStarForest}. 
\end{proof}

\cref{ExcludeStar,NoInducedP3Forest,Induced-UTW2} prove (iii) in \cref{InducedSubgraphUnderlyingTreewidth}.

\begin{lem}
\label{Induced-UTW0}
If $E(H)=\emptyset$ or $H$ is a path on at most three vertices, then $\II_H$ has underlying treewidth $0$.
\end{lem}

\begin{proof}
First suppose that $E(H)=\emptyset$. 
For each $G\in\II_H$, 
$\{V(G)\}$ is a 0-tree-partition of $G$ with width $\abs{V(G)}\leq \chi(G)\alpha(G) \leq (\tw(G)+1)(\abs{V(H)}-1)$.
Hence $\II_H$ has underlying treewidth 0.
Now suppose $H$ is a path on at most three vertices.
For each $G\in \II_H$, if $G_1,\dots,G_k$ are the connected components of $G$, then each $G_i$ is a complete graph, implying $\{V(G_1),\dots,V(G_k)\}$ is a 0-tree-partition of $G$ with width $\max_i\abs{V(G_i)} = \tw(G)+1$.
\end{proof}

\begin{lem}
\label{Induced-UTW-AtLeast1}
If $H$ is a star-forest with at least one edge and at least two components, then $\II_H$ has underlying treewidth at least $1$. 
\end{lem}

\begin{proof}
Let $G \coloneqq K_{2,n}$. Then for every edge $vw$ of $G$, every vertex $x\in V(G)\setminus\{v,w\}$ is adjacent to $v$ or $w$. Thus $G\in \II_H$. Since $G$ is connected, $\tpw_0(G)=n+2$. Since $\tw(G)\leq 2$, $\II_H$ has underlying treewidth at least $1$.
\end{proof}

\cref{Induced-UTW0,Induced-UTW-AtLeast1} prove (iv) in \cref{InducedSubgraphUnderlyingTreewidth}.

We finish this section with an open problem. For any set $\XX$ of graphs, let $\II_\XX$ be the class of graphs $G$ such that no induced subgraph of $G$ is isomorphic to a graph in $\XX$. For which sets $\XX$ does $\II_\XX$ have bounded underlying treewidth? Consider the case when $\XX$ is finite. If some star-forest is in $\XX$, then $\II_\XX$ has bounded underlying treewidth by \cref{ExcludeStarForest}. If no star-forest is in $\XX$ and $\II_\XX$ still has bounded underlying treewidth, then some spider-forest is in $\XX$ (by \cref{ExcludingSubgraph}), and some graph for which every component is the closure of a rooted tree is also in $\XX$ (by \cref{StandardExampleUnderlyingTreewidth}). In related work, \citet{LR22} proved a graph class defined by finitely many excluded induced subgraphs has bounded treewidth if and only if it excludes a complete graph, a complete bipartite graph, a tripod (a forest in which every connected component has at most three leaves), and the line graph of a tripod.

\section{Graph Drawings}
\label{GraphDrawings}

A graph is \defn{$k$-planar} if it has a drawing in the plane with at most~$k$ crossings on each edge, where we assume that no three edges cross at the same point. 
Of course, the class of $0$-planar graphs is the class of planar graphs, which has underlying treewidth~3 (\cref{planarunderlying}). 
However, 1-planar graphs behave very differently. 
It is well-known that every graph has a 1-planar subdivision: 
take an arbitrary drawing of~$G$ and for each edge~$e$ add a subdivision vertex between consecutive crossings on~$e$. 
Since the class of 1-planar graphs is monotone, \cref{MonotoneUnderlyingTreewidth} implies that the class of 1-planar graphs has unbounded underlying treewidth. 

By restricting the type of drawing, we obtain positive results. 
A \defn{circular drawing} of a graph~$G$ positions each vertex on a circle in the plane, and draws each edge as an arc across the circle, such that no two edges cross more than once. 
A graph is \defn{outer $k$-planar} if it has a circular drawing such that each edge is involved in at most~$k$ crossings. 
The outer 0-planar graphs are precisely the outer-planar graphs, which have treewidth~$2$. 
We show below that for each~${k \in \NN}$, the class of outer $k$-planar graphs has underlying treewidth~$2$. 
In fact, we prove a slightly more general result. 
A graph is \defn{weakly outer $k$-planar} if it has a circular drawing such that whenever two edges~$e$ and~$f$ cross, $e$ or~$f$ crosses at most~$k$ edges. 
Clearly every outer $k$-planar graph is weakly outer $k$-planar.

\begin{thm}
\label{OuterkPlanarthm}
    Every weakly outer $k$-planar graph~$G$ has $2$-tree-partition-width~${O(k^3)}$. 
\end{thm}

\begin{proof}
    \citet[Prop.~8.5]{WT07} proved that any weakly outer $k$-planar graph $G$ has ${\tw(G) \leq 3k+11}$. 
    Thus, by \cref{MainCorollary}, it suffices to show that the singleton partition of~$G$ is ${(2,4k+4)}$-disjointed. 
    Let~${v_1,\dots, v_n}$ be the cyclic ordering of~${V(G)}$ given by the drawing. 
    Addition is taken modulo~$n$ in this proof. 
    If~${n \leq 2}$, then the claim holds trivially, so assume that~${n \geq 3}$. 
    We may assume that~${v_iv_{i+1} \in E(G)}$ for all~${i \in \{1,\dots,n\}}$. 
    
    Consider distinct~${v_i,v_j \in V(G)}$. 
    If~${j \neq i+1}$, then let~$v_r'$ be the clockwise-closest vertex to~$v_i$ that is adjacent to~$v_j$ and let~$v_s'$ be the anti-clockwise-closest vertex to~$v_i$ that is adjacent to~$v_j$. 
    If~${v_r'v_j}$ is involved in more than~$k$ crossings, then let~$v_r$ be the anticlockwise-closest vertex to~$v_j$ that is incident to an edge~$e_r$ that crosses~$v_r'v_j$. 
    Otherwise~${v_r'v_j}$ is involved in at most~$k$ crossings so let~${v_r \coloneqq v_r'}$ and~${e_r \coloneqq v_r'v_j}$. 
    Similarly, if~${v_s'v_j}$ is involved in more than~$k$ crossings, then let~$v_s$ be the clockwise-closest vertex to~$v_j$ that is incident to an edge~$e_s$ that crosses~${v_s'v_j}$. 
    Otherwise let~${v_s \coloneqq v_s'}$ and~${e_s \coloneqq v_s'v_j}$. 
    Let~$R$ be the set of vertices that are incident to edges that cross~$e_r$ and let~$S$ be the set of vertices incident to edges that cross~$e_s$. 
    Let~${Q \coloneqq (R \cup S \cup \{v_r,v_s,v_r',v_s'\}) \setminus \{v_i,v_j\}}$. 
    The setup so far is illustrated in \cref{OuterkPlanar}, with vertices in~$Q$ circled. 
    We have~${\abs{Q} \leq 4k+4}$, since~$e_r$ and~$e_s$ are involved in at most~$k$ crossings. 
    
    Let ${a \in N(v_i) \setminus (Q\cup\{a,b\})}$, ${b \in N(v_j) \setminus (Q \cup \{a,b\})}$, and $P$ be an ${(a,b)}$-path in~${G - \{v_i,v_j\}}$. 
    Since~$G$ is weakly outer $k$-planar, if~$P$ does not contain~$v_r$ or~$v_s$ as an internal vertex, then it contains an edge that crosses~$e_r$ or~$e_s$. 
    Thus~${V(P) \cap Q \neq \emptyset}$ and hence~$Q$ separates~${N(v_i)}$ from~${N(v_j)}$ in~${G - \{v_i,v_j\}}$. 
    As such, $G$ is ${(2,4k+4)}$-disjointed.
\end{proof}

\begin{figure}[ht]
    \centering
    \includegraphics{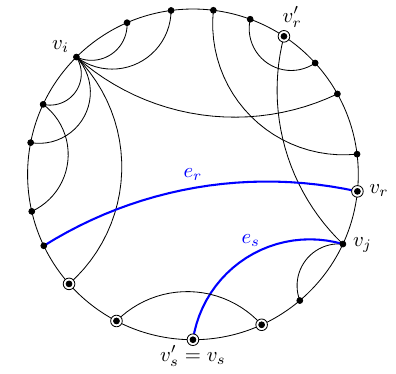}
    \caption{Setup to define~$Q$ (circled vertices) with~${k = 2}$.}
    \label{OuterkPlanar}
\end{figure}

\cref{OuterkPlanarthm} implies the next result, where the lower bound holds since~$G_{2,\ell}$ from \cref{StandardExamples} is outerplanar.

\begin{thm}
    \label{OuterkPlanarUnderlyingTreewidth}
    For every fixed~${k \in \NN}$, the underlying treewidth of the class of weakly outer $k$-planar graphs equals~2, with constant treewidth-binding function. 
\end{thm}

A graph~$D$ is a \defn{planarisation} of a graph~$G$ if~$G$ has a drawing in the plane and~$D$ is the graph obtained by replacing each crossing point with a dummy vertex. Observe that~$G$ is a topological minor of~${D \boxtimes K_2}$. 
As such, we have the following consequence of \cref{cTreePartitionWidthTreewidth,TopologicalMinor,planarunderlying}.

\begin{cor}
    \label{TopologicalMinorConsequence}
    For every graph~$G$, if~$D$ is a planarisation of~$G$ where~${\tw(D) < k}$ then~$G$ has $3$-tree-partition-width~${O(k^{9})}$. 
\end{cor}

As a special case of \cref{TopologicalMinorConsequence}, consider a drawing of a graph~$G$ with at most~$k$ crossings per edge. 
If~$G$ has radius~$r$, then the planarisation of~$G$ has radius at most~${r(k+1)}$ and thus has treewidth at most~${3r(k+1)}$ \cite[(2.7)]{RS-III}. 
Thus we have the following:

\begin{cor}
    \label{radiuskplanar}
    If $G$ is a $k$-planar graph with radius $r$, then $G$ has $3$-tree-partition-width $O((rk)^{9})$.
\end{cor}

\subsection*{Acknowledgements} This research was initiated at the \href{https://www.matrix-inst.org.au/events/structural-graph-theory-downunder-ll/}{Structural Graph Theory Downunder II} workshop at the Mathematical Research Institute MATRIX (March 2022).

{
\fontsize{10pt}{11pt}
\selectfont
\let\oldthebibliography=\thebibliography
\let\endoldthebibliography=\endthebibliography
\renewenvironment{thebibliography}[1]{%
\begin{oldthebibliography}{#1}%
\setlength{\parskip}{0.25ex}%
\setlength{\itemsep}{0.25ex}%
}{\end{oldthebibliography}}
 \bibliographystyle{DavidNatbibStyle}
 \bibliography{DavidBibliography}
}

\newpage
\appendix
\section{Omitted Proofs}
\label{Omitted}

Here we provide complete proofs of results mentioned above.

\cTreePartitionWidthTreewidth*

\begin{proof}
    Let ${(V_h \colon h \in V(H))}$ be a $c$-tree-partition of~$G$ with width ${\tpw_c(G)}$. 
    Let ${(W_x \colon x \in V(T))}$ be a tree-decomposition of~$H$ of width $c$. 
    For each~${x \in V(T)}$, let ${U_x \coloneqq \bigcup\{V_h \colon h \in W_x\}}$. 
    We now show that ${(U_x \colon x \in V(T))}$ is a tree-decomposition of $G$. 
    For each edge $vw$ of $G$, if ${v \in V_h}$ and~${w \in V_j}$, then~${h = j}$ or~${hj \in E(H)}$; 
    in both cases, there exists~${x \in V(T)}$ such that ${j,h \in W_x}$, implying ${v,w \in U_x}$. 
    For each vertex~$v$ of~$G$, if~${v \in V_h}$, then ${\{x \in V(T) \colon v \in U_x\}} = {\{ x \in V(T) \colon h \in W_x \}}$; 
    since the latter set induces a subtree of~$T$, so does the former. 
    Thus~${(U_x \colon x \in V(T))}$ is a tree-decomposition of~$G$. 
    By construction, the width is at most~${(c+1) \tpw_c(G) - 1}$. 
\end{proof}

\citet[(8.7)]{RS-V} state that the following lemma is a ``standard result in hypergraph theory''.

\begin{lem}\label{EP-subtrees}
    Let $\FF$ be a set of subtrees of a tree $T$. For every non-negative integer $\ell$, either there are $\ell + 1$ vertex-disjoint trees in $\FF$ or there is a set $S \subseteq V(T)$ of size at most $\ell$ such that $T' \cap S \neq \emptyset$ for all $T' \in \FF$.
\end{lem}

\begin{proof}
Let $G$ be the intersection graph
of $\FF$. If $\alpha(G)\geq\ell+1$ then 
there are $\ell + 1$ vertex-disjoint trees in $\FF$. So we may assume that $\alpha(G) \leq \ell$. Since $G$ is chordal and thus perfect, there is a partition $X_1,\dots,X_\ell$ of $V(G)$ into cliques in $G$. For each $i$, the subtrees in $X_i$ pairwise intersect. By the Helly property, there is a node $x_i$ in every subtree in $X_i$. Thus $S\coloneqq \{x_1,\dots,x_\ell\}$ is the desired set.
\end{proof}

\cref{EP-treewidth} follows from \cref{EP-subtrees}.

\EPtreewidth*

\begin{proof}
    Let $(T, \WW)$ be a tree-decomposition of $G$ with width $\tw(G)$. For each $H \in \HH$, let $T_H$ be the subtree of $T$ induced by $\{x \in V(T) \colon V(H) \cap W_x \neq \emptyset\}$. Since each $H$ is connected, each $T_H$ is a subtree of $T$. Let $\FF = \{T_H \colon H \in \HH\}$. If $\FF$ contains $\ell + 1$ vertex-disjoint trees, then the corresponding $\ell + 1$ graphs in $\HH$ are vertex-disjoint. Hence, by \cref{EP-subtrees}, there is a set $S \subseteq V(T)$ of size at most $\ell$ such that $T_H \cap S \neq \emptyset$ for all $H \in \HH$. Let $Q = \cup_{x \in S} W_x$. This has size at most $\ell(\tw(G) + 1)$ and $Q \cap V(H) \neq \emptyset$ for all $H \in \HH$, as required.
\end{proof}

\begin{prop}
    \label{MinorClosedClass0}
    A minor-closed class~$\GG$ has underlying treewidth~$0$ if and only if~$\GG$ has clustered chromatic number~$1$. 
\end{prop}

\begin{proof}
    Say~$\GG$ has clustered chromatic number~$1$. 
    Then there exists~${\ell \in \NN}$ such that every connected graph in~$\GG$ has at most~$\ell$ vertices. Hence~$\GG$ has underlying treewidth~$0$ with treewidth-binding function~${f(k) = \ell}$. 

    Conversely, suppose that~$\GG$ has underlying treewidth~$0$ with treewidth-binding function~$f$, and consider a spanning tree~$T$ of a non-trivial component of a graph in~$\mathcal{G}$. 
    Since~${T \in \mathcal{G}}$ and~${\tw(T) = 1}$, we have~${\abs{V(T)} \leq f(1)}$.
    As such, every non-trivial component of any graph in~$\GG$ has at most~${f(1)}$ vertices and thus~$\GG$ has clustered chromatic number at most~$1$ with clustering~${\max\{1,f(1)\}}$. 
\end{proof}
 
\begin{prop}
    \label{MinorClosedClass1}
    A minor-closed class~$\GG$ has underlying treewidth at most~$1$ if and only if~$\GG$ has clustered chromatic number at most~$2$.
\end{prop}

\begin{proof}
    Say~$\GG$ has clustered chromatic number at most~$2$. 
    By a result of \citet{NSSW19}, for some~${k \in \NN}$, the $k$-fan, the $k$-fat path and the $k$-fat star are not in~$\GG$. 
    By a result of \citet{DO96}, $\GG$ has underlying treewidth at most~$1$. 
    Conversely, suppose that~$\GG$ has underlying treewidth~$1$ with treewidth-binding function~$f$. 
    By \cref{StandardExamples}, some fan graph is not in~$\GG$. 
    Since fan graphs are planar, $\GG$ has bounded treewidth, say at most~$k$. 
    Hence every graph in~$\mathcal{G}$ is contained in~${T \times K_{\ell}}$ for some tree~$T$ and~${\ell \coloneqq \max\{f(1),\dots,f(k)\}}$, and is therefore 2-colourable with clustering~${\ell}$.
\end{proof}

\begin{prop}
    \label{ClusteredApplication}
    Every minor-closed graph class~$\GG$ with underlying treewidth~$c$ has clustered chromatic number at most~${2(c+1)}$.
\end{prop}

\begin{proof}
    Let~$f$ be the treewidth-binding function for~$\GG$. 
    \citet{DDOSRSV04} showed that, for some~${w = w(\GG)}$, for every graph~${G \in \GG}$ there is a partition~${V_1,V_2}$ of~${V(G)}$ such that~${\tw(G[V_i]) \leq w}$ for~${i \in \{1,2\}}$; see \citep{DJMMUW20} for an alternative proof based on graph products. 
    Since~$\GG$ is minor-closed, ${G[V_1], G[V_2] \in \GG}$. 
    Since~$\GG$ has underlying treewidth at most~$c$, for some graphs~$H_1$ and~$H_2$ of treewidth at most~$c$, we have~${G[V_1]}$ is contained in~${H_1 \boxtimes K_\ell}$ and~${G[V_2]}$ is contained in~${H_2 \boxtimes K_\ell}$ where~${\ell \coloneqq \max\{f(0),\dots,f(w)\}}$. 
    Thus~${G[V_1]}$ and~${G[V_2]}$ can both be~${c+1}$ coloured with clustering~${\ell}$. 
    Using distinct colours for~${G[V_1]}$ and~${G[V_2]}$, the graph~$G$ can be~${2(c+1)}$ coloured with clustering~${\ell}$, as required. 
\end{proof}

\begin{prop}
    \label{K1tMinorPartition}
    For~${t \in \NN}$, every $K_{1,t}$-minor-free graph~$G$ has tree-partition-width~${O(t)}$. 
\end{prop}

\begin{proof}
    We may assume that~$G$ is connected. 
    Observe that~$G$ has no spanning tree with~$t$ leaves. 
    It follows that~$G$ is a subdivision of a graph~$H$ on at most~${10t}$ vertices; see \citep{DJS01}. 
    Using the method in the proof of \cref{monotonesubdivision}, it is easy to construct a tree-partition of~$G$ with width at most~${10t}$. 
\end{proof}

\begin{prop}
\label{K2tTopoMinorPartition}
    For~${t \in \NN}$, every $K_{2,t}$-topological minor-free graph~$G$ of treewidth~$k$ has $2$-tree-partition-width~${O(t^2k)}$. 
\end{prop}

\begin{proof}
    By \cref{Main} it suffices to show that the singleton partition of~$G$ is $(2,f)$-disjointed, where~${f(n) \in O(t^2n)}$. 
    To this end, 
    let~${S_1,S_2}$ be subsets of~${V(G)}$ of size at most~$n$, let~${S \coloneqq S_1 \cup S_2}$, let~${S'_1 \coloneqq S_1}$ and let~${S'_2 \coloneqq S_2 \setminus S_1}$. 
    Let~${A_i \coloneqq N_G(S'_i) \setminus S}$. 
    Let~$\HH$ be a maximum-sized set of pairwise disjoint $({A_1,A_2)}$-paths in~${G-S}$. 
    For each path~${P \in \HH}$ there are vertices~${x_1 \in S'_1}$ and~${x_2 \in S'_2}$ respectively adjacent to the end-vertices of~$P$. 
    Charge~$P$ to~${(x_1,x_2)}$. 
    Since~$G$ is $K_{2,t}$-topological-minor-free, at most~${t-1}$ paths in~$\HH$ are charged to each pair~${(x_1,x_2)}$ with~${x_i \in S'_i}$. 
    Let~$G'$ be the graph with vertex-set~$S$, where~${x_1x_2 \in E(G')}$ if some path in~$\HH$ is charged to~${(x_1,x_2)}$. 
    Thus~$G'$ is a topological-minor of~$G$, and~$G'$ is $K_{2,t}$-topological-minor-free. 
    Thus~${\abs{E(G')} \leq \gamma t \abs{S}}$ for some absolute constant~$\gamma$; see \citep{ReedWood16}. 
    Hence~${\abs{\HH} \leq \gamma t^2\abs{S}}$. 
    Let~${f(n) \coloneqq \gamma t^2 2 n}$. 
    By Menger's Theorem, there is a set~${Q \subseteq V(G-S)}$ of size at most~${\gamma t^2\abs{S} \leq f(n)}$, such that there is no ${(A_1,A_2)}$-path in $G-(S\cup Q)$. 
    So every component~$Y$ of~${G-(S\cup Q)}$ satisfies~${Y \cap N_{G}(S'_i) = \emptyset}$ for some~${i \in \{1,2\}}$. 
    Hence the singleton partition of~$G$ is $(2,f)$-disjointed.
\end{proof}

\begin{prop}
    \label{K3tTopoMinorPartition}
    For~${t \in \NN}$, every $K_{3,t}$-topological minor-free graph~$G$ of treewidth~$k$ has $3$-tree-partition-width~${O(t^3 k^2)}$.
\end{prop}

\begin{proof}
    By \cref{Main} it suffices to show that the singleton partition of~$G$ is ${(3,f)}$-disjointed, 
    where~${f(n) \in O(k t^3 n)}$. To this end,
    let~$S_1, S_2, S_3$ be subsets of~${V(G)}$ of size at most~$n$, let~${S \coloneqq S_1 \cup S_2 \cup S_3}$, and for~${i \in \{1,2,3\}}$ let~${S'_i \coloneqq S_i \setminus (S_1 \cup \dots \cup S_{i-1})}$. 
    Let~$\HH$ be the set of connected subgraphs~$H$ of~${G-S}$ such that~${V(H) \cap N_G(S'_i) \neq \emptyset}$ for each~${i \in \{1,2,3\}}$. 
    
    Consider any set~$\JJ$ of pairwise vertex-disjoint graphs in~$\HH$. 
    Assign subgraphs of~$\JJ$ to pairs of vertices in~$S$ as follows. 
    Initially, no subgraphs of~$\JJ$ are assigned. 
    If there is an unassigned subgraph~${H \in \JJ}$ adjacent to vertices~${x \in S'_i}$ and~${y \in S'_j}$, for some distinct~${i,j \in \{1,2,3\}}$, and no subgraph in~$\JJ$ is already assigned to~${\{x,y\}}$, then assign~$H$ to~${\{x,y\}}$. 
    Repeat this operation until no more subgraphs in~$\JJ$ can be assigned. 
    
    Let~$\JJ_1$ and~$\JJ_2$ be the sets of assigned and unassigned subgraphs in~$\JJ$ respectively. 
    Let~$G'$ be the graph obtained from~$G$ as follows: 
    for each~${H \in \JJ_1}$ assigned to~$\{x,y\}$, contract an ${(x,y)}$-path through~$H$ down to a single edge~$xy$. 
    Delete any remaining vertices not in~$S$. 
    Thus~$G'$ is a topological minor of~$G$ with~${V(G') = S}$ and~${\abs{E(G')} \geq \abs{\JJ_1}}$. 
    Consider a subgraph~${H \in \JJ_2}$. 
    Since~${H \in \HH}$ there are neighbours~${v_1,v_2,v_3}$ of~$H$ with~${v_i \in S'_i}$ for each~${i \in \{1,2,3\}}$. 
    Since~${H \in \JJ_2}$, the graph~$G'[\{v_1,v_2,v_3\}]$ is a triangle (otherwise~$H$ could have been assigned to some non-adjacent~$v_i$ and~$v_j$). 
    Charge~$H$ to~${(v_1,v_2,v_3)}$. 
    
    Suppose that~${X_1,\dots,X_t}$ are distinct subgraphs in~$\JJ_2$ charged to $(v_1,v_2,v_3)$ for some triangle~${G'[\{v_1,v_2,v_3\}]}$ with~${v_i \in S'_i}$. 
    Fix~${j \in \{1,\dots,t\}}$. 
    Let~$w_i$ be a neighbour of~$v_i$ in~$X_j$. 
    Let~$Y_j$ be a minimal connected subgraph of~$X_j$ including~${w_1,w_2,w_3}$. 
    Thus~$Y_j$ is a tree, in which each leaf is one of~$w_1$, $w_2$ or~$w_3$, which are not necessarily distinct vertices. 
    Hence~$Y_j$ is a subdivision of a star with at most three leaves. 
    If~$Y_j$ has exactly three leaves, then let~$x_j$ be the central vertex of this star. 
    Otherwise~$Y_j$ is a path. 
    In this case, without loss of generality, the endpoints of~$Y_j$ are~$w_1$ and~$w_2$ and we set~${x_j \coloneqq w_3}$. 
    
    Now, ${Y_1 \cup \dots \cup Y_t}$ together with the edges~${v_1w_1}$, ${v_2w_2}$ and~${v_3w_3}$ form a $K_{3,t}$-topological minor with branch vertices~${v_1,v_2,v_3,x_1,\dots,x_t}$. 
    This contradiction shows that there are at most~${t-1}$ subgraphs in~$\JJ_2$ charged to~$(v_1,v_2,v_3)$ for each triangle~${G'[\{v_1,v_2,v_3\}]}$ with~${v_i \in S'_i}$. 
    
    Since $G'$ is $K_{3,t}$-topological-minor-free, $\abs{E(G')} \leq \gamma t \abs{S}$ for some constant $\gamma$; see \citep{ReedWood16}. 
    Thus ${\abs{\JJ_1} \leq \gamma t \abs{S}}$. 
    Moreover, $G'$ is $2\gamma t$-degenerate. 
    Every $d$-degenerate graph on~$n$ vertices has at most $\binom{d}{2}n$ triangles~\citep[Lemma 18]{Wood16}.
    Thus $\abs{\JJ_2} \leq (t-1) \binom{2\gamma t}{2}\abs{S}$. 
    Hence $\abs{\JJ} = \abs{\JJ_1} + \abs{\JJ_2} < 2 \gamma^2 t^3 \abs{S} \leq 6 \gamma^2 t^3n$. 
    Define~$f(n) \coloneqq (k+1) 6 \gamma^2 t^3 n$. By \cref{EP-treewidth}, there is a set~${Q \subseteq V(G-S)}$ of size at most~${f(n)}$ such that no subgraph of~${G-(S\cup Q)}$ is in~$\HH$. 
    So every component~$Y$ of~${G-(S\cup Q)}$ satisfies~${Y \cap N_{G}(S'_i) = \emptyset}$ for some~${i \in \{1,2,3\}}$. 
    Hence the singleton partition of~$G$ is $(3,f)$-disjointed.
\end{proof}

\JournalArxiv{}{
\section{Other Underlying Parameters}
\label{UnderlyingParameters}

It is natural to consider other underlying parameters\footnote{A graph parameter is a function $\theta$ such that $\theta(G)\in \REALS$ for every graph $G$ and $\theta(G_1)=\theta(G_2)$ for all isomorphic graphs $G_1$ and $G_2$}. For a graph parameter $\theta$ and a graph class $\GG$, the \defn{underlying $\theta$} of $\GG$ is the minimum $c\in\NN$ such that, for some function $f$, every graph $G\in \GG$ is contained in $H\boxtimes K_{f(\theta(G))}$ for some graph $H$ where $\theta(H)\leq c$. Equivalently, $G$ has an $H$-partition of width at most $f(\theta(G))$ where $\theta(H)\leq c$. 

First we consider some parameters closely related to treewidth. 
A \defn{path-decomposition} is a $P$-decomposition for any path~$P$. 
The \defn{pathwidth $\pw(G)$} of a graph~$G$ is the minimum width of a path-decomposition of $G$. The \defn{bandwidth $\bw(G)$} of a graph $G$ is the minimum integer $k$ such that there is a linear order $v_1,\dots,v_n$ of $V(G)$ such that $\abs{i - j} \leq k$ for each edge $v_i v_j\in E(G)$. 
It is well-known that $\tw(G)\leq\pw(G)\leq\bw(G)$ for every graph~$G$. 
The \defn{treedepth} $\td(G)$ of a graph $G$ is the minimum vertex-height of a rooted tree $T$ such that $G$ is a subgraph of the closure of $T$, where the \defn{vertex-height} of a rooted tree is the maximum number of vertices in a root--leaf path. It is well-known that $\tw(G)\leq\pw(G)\leq\td(G)-1$ for every graph~$G$. 

Some underlying parameters are trivial. For example, it is well-known that every graph of bandwidth $k$ is contained in $P \boxtimes K_k$ where $P$ is a path (with bandwidth 1). Thus the class of all graphs has underlying bandwidth $1$. 
Since connected graphs with bounded treedepth and bounded degree have a bounded number of vertices, any class of graphs with bounded maximum degree has underlying treedepth 0. 

\cref{TreewidthUnderlyingTreewidth} shows that the underlying treewidth of the class of graphs of treewidth at most $k$ equals $k$. Similarly, the underlying treewidth of the class of pathwidth $k$ graphs equals $k$, and the underlying treewidth of the class of treedepth $k+1$ graphs equals $k$ (since $\pw(C_{c, \ell}) = \td(C_{c, \ell}) -1= c$ and $\tw(G)\leq\pw(G)\leq\td(G)-1$ for every graph $G$).

We now show that trees have unbounded underlying pathwidth and unbounded underlying treedepth. This is implied by the following result, since the pathwidth of a tree is at most its treedepth minus 1.

\begin{prop}\label{UnderlyingPathwidth}
For all $h\in\NN_0$ and $t\in\NN$, there exists a tree $T$ with radius at most $h$ such that if $T$ is contained in $H \boxtimes K_t$ then $\pw(H)\geq h$ and $\td(H)\geq h+1$. 
\end{prop}

\begin{proof}
Let $(X_h,\tilde{x})$ be the complete ternary tree with height $h$ rooted at $\tilde{x}$. We prove the following induction hypothesis: for all $h\in\NN_0$ and $t\in\NN$, there exists a rooted tree $(T_h,r)$ where $\dist_{T_h}(v,r)\leq h$ for all $v \in V(T_{h})$ such that for every $H$-partition $(V_x \colon x\in V(H))$ of $T_h$ with width at most $t$, $H$ contains $(X_h,\tilde{x})$ as a subgraph with $r\in V_{\tilde{x}}$. Since $\pw(X_h)=h$ and
$\td(X_h)=h+1$, this implies the claim. 
    
We proceed by induction on $h$. For $h=0$, the claim holds trivially by setting $(T_0,r)$ to be a single vertex $r$. 
    
Now suppose $h>0$. Let $(T_{h-1},r')$ be given by the induction hypothesis and let $n = \abs{V(T_{h-1})}$. Let $m=(2n+1)t$ and let $(T_{h},r)$ be obtained by taking $m$ copies $((T_{h-1,1},r_1),\dots, (T_{h-1,m},r_{m}))$ of $(T_{h-1},r')$ plus a vertex $r$ along with the edges $rr_i$ for all $i\in \{1,\dots,m\}$. Since $\dist_{T_{h-1,i}}(v_i,r_i)\leq h-1$ for all $v_i\in V(T_{h-1,i})$, we have $\dist_{T_{h}}(v,r)\leq h$ for all $v\in V(T_{h})$. 
    
    Let $(V_x \colon x\in V(H))$ be an $H$-partition of $T_h$ with width at most $t$. Let $\tilde{x}\in V(H)$ be such that $r\in V_{\tilde{x}}$. Let $X \coloneqq \{i \in \{1,\dots,m\} \colon V(T_{h-1,i})\cap V_{\tilde{x}}=\emptyset\}$. Since $\abs{V_{\tilde{x}}} \leq t$ it follows that $\abs{X} \geq m-(t-1)\geq 2nt+1$. Choose any $j\in X$ and let $A\coloneqq\{x\in V(H) \colon V_x\cap V(T_{h-1,j})\neq \emptyset\}$. Let $Y\coloneqq\{i \in X \colon V(T_{h-1,i})\cap (\bigcup_{x\in A}V_x)=\emptyset\}$. Since $\abs{A}\leq n$, it follows that $\abs{Y} \geq \abs{X}-nt\geq nt+1$. Choose any $k\in Y$ and let $B\coloneqq\{x\in V(H) \colon V_x\cap V(T_{h-1,k})\neq \emptyset\}$. Let $Z \coloneqq \{i \in Y \colon V(T_{h-1,i}) \cap (\bigcup_{x\in B} V_x) = \emptyset\}$. As before, $\abs{Z} \geq \abs{Y}-nt\geq 1$. Choose any $\ell \in Z$ and let $C \coloneqq \{x \in V(H) \colon V_x\cap V(T_{h-1,\ell})\neq \emptyset\}$. 
    
    By construction, $\{r\}, A,B$ and $C$ are pairwise disjoint. Let $(i,I)\in \{(j,A),(k,B),(\ell,C)\}$. Since $(V_x\cap V(T_{h-1,i}) \colon x\in I)$ is a partition of $T_{h-1,i}$, it follows by induction that $H[I]$ contains $(H_{h-1,i},x_{i})$ as a subgraph where $r_{i}\in V_{x_{i}}$. Thus $H[ \{\tilde{x}\}\cup A \cup B \cup C]$ contains the desired complete ternary tree.
\end{proof}

We now consider the underlying chromatic number of a graph class, which is closely related to its clustered chromatic number (defined in \cref{LowerBounds}). 

\begin{thm}\label{uchi-cchi}
Let $\GG$ be a graph class.
\begin{enumerate}[label = \textnormal{(\roman*)}]
    \item The underlying chromatic number of $\GG$ is at most the clustered chromatic number of $\GG$.
    \item There is a function $f$ such that if $\GG$ has underlying chromatic number at most $c$, then every $G \in \GG$ can be $c$-coloured with clustering $f(\chi(G))$. 
    \item In particular, if $\GG$ has bounded chromatic number, then the underlying chromatic number of $\GG$ is equal to the clustered chromatic number of $\GG$.
\end{enumerate}
\end{thm}

\begin{proof}
We first prove (i). If $\GG$ has unbounded clustered chromatic number, then there is nothing to prove. Otherwise we may suppose that $\GG$ has clustered chromatic number $c$. Then for some $k \in \NN$ every graph $G \in \GG$ is $c$-colourable with clustering $k$. So each $G \in \GG$ is contained in $H \boxtimes K_k$ for some complete $c$-colourable graph $H$. This proves (i).

Now suppose that $\GG$ has underlying chromatic number $c$. That is, there is a function $f$ such that every $G \in \GG$ is contained in $H \boxtimes K_{f(\chi(G))}$ for some $c$-colourable $H$. A $c$-colouring of $H$ determines a $c$-colouring of $G$ with clustering $f(\chi(G))$. This proves (ii).

Now suppose that $\GG$ has bounded chromatic number, $\chi(\GG)$. Then $\GG$ has bounded underlying chromatic number $c$. By (ii), every $G \in \GG$ can be $c$-coloured with clustering $\max\{f(1), \dotsc, f(\chi(\GG))\}$ and so $\GG$ has clustered chromatic number at most $c$. Together with (i) this proves (iii).
\end{proof}

Despite (i) and (iii), the underlying and clustered chromatic numbers of a graph class are not necessarily equal. Trivially, the class of complete graphs has underlying chromatic number $1$ but unbounded clustered chromatic number. Our next result gives a variety of classes with unbounded underlying chromatic number.

A \defn{blow-up} of a graph $G$ is a graph obtained by replacing each vertex $v$ of $G$ by an independent set $I_v$ and each edge $uv$ by a complete bipartite graph between $I_u$ and $I_v$. Let $G(t)$ be the blow-up of $G$ where $\abs{I_v} = t$ for each $v\in V(G)$.
As an example, complete $r$-partite graphs are exactly the blow-ups of $K_r$. A family of graphs $\GG$ is \defn{closed under taking blow-ups} if for every $G \in \GG$ every blow-up of $G$ is also in $\GG$.

\begin{thm}\label{uchi-blowups}
    For every graph class $\GG$, if $\GG$ is closed under taking blow-ups and has unbounded chromatic number, then $\GG$ has unbounded underlying chromatic number.
\end{thm}

\begin{proof}
    Suppose that $\GG$ has underlying chromatic number at most $c$. Since $\GG$ has unbounded chromatic number it contains some graph $H$ with $k \coloneqq \chi(H) > c$. Every blow-up of $H$ has chromatic number $k$. By \cref{uchi-cchi} (ii), every blow-up of $H$ can be $c$-coloured with clustering $f(k)$.
    Consider $G \coloneqq H(c\,f(k))$ and denote its parts by $V_1, \dotsc, V_{\abs{H}}$. As $\GG$ is closed under taking blow-ups, $G \in \GG$. Consider a $c$-colouring of $G$. Since $\abs{V_i} = c\,f(k)$, there exists a monochromatic subset $V'_i \subset V_i$ of size $f(k)$. The colouring of $G[V'_1 \cup \dotsb \cup V'_{\abs{V(H)}}]$ corresponds to a $c$-colouring of $H$. Since $\chi(H) > c$, there exists $i, j$ such that $V'_i$, $V'_j$ are adjacent parts and have the same colour. Then $V'_i \cup V'_j$ is a monochromatic component of size $2f(k) > f(k)$. Hence, $G$ has no $c$-colouring with clustering $f(k)$, which is a contradiction.
\end{proof}

Recall that a graph is \defn{$H$-free} if it contains no (not necessarily induced) subgraph isomorphic to $H$ and we denote the class of $H$-free graphs by $\GG_H$. For general $H$, the class $\GG_H$ is not necessarily closed under taking blow-ups (for example, $K_3$ is $C_5$-free but $K_3(2)$ is not). However, $\GG_H$ does contain a very natural subfamily that is closed under taking blow-ups. Recall that a function $\varphi \colon V(H) \to V(G)$ is a \defn{graph homomorphism} if it is edge-preserving. There is a graph homomorphism $H \to G$ if and only if $H$ is a subgraph of some blow-up of $G$. Let $\GG_{H \not\to}$ be the class of graphs \defn{$H$-homomorphism-free graphs}: those graphs $G$ for which there is no homomorphism from $H$ to $G$. This is, equivalently, the class of graphs whose blow-ups are all $H$-free. Now $\GG_{H\not\to}$ is closed under taking blow-ups and is a subfamily of $\GG_H$. 

\begin{thm}\label{uchi-Hhomfree}
    For any graph $H$, the class $\GG_{H\not\to}$ of $H$-homomorphism-free graphs has unbounded underlying chromatic number if and only if $H$ is not bipartite.
\end{thm}

\begin{proof}
    If $H$ is bipartite, then there is a homomorphism $H \to K_2$ and so $\GG_{H\not\to}$ consists of the graphs with no edges and hence has underlying chromatic number 1. If $H$ is not bipartite, then $H$ contains some odd cycle. Let the length of this odd cycle be $\ell$. Any graph with girth greater than $\ell$ is in $\GG_{H\not\to}$. There are graphs of arbitrarily large chromatic number and girth~\citep{Erdos59}. Thus $\GG_{H\not\to}$ has unbounded chromatic number, and so, by \cref{uchi-blowups}, unbounded underlying chromatic number.
\end{proof}

Since $\GG_{H\not\to}$ is a subfamily of $\GG_H$, \cref{uchi-Hhomfree} immediately implies that $\GG_H$ has unbounded underlying chromatic number if $H$ is not bipartite. In fact, we can do better.

\begin{thm}\label{uchi-Hfree}
For any graph $H$, the class $\GG_H$ of $H$-free graphs has bounded underlying chromatic number if and only if $H$ is a forest.
\end{thm}

This will follow from the following more general result.

\begin{thm}\label{uchi-highgchi}
Let $\GG$ be a graph class such that, for infinitely many $k$, $\GG$ contains graphs with chromatic number $k$ and arbitrarily high girth. Then $\GG$ has unbounded underlying chromatic number.
\end{thm}

\begin{proof}
    Suppose that $\GG$ has underlying chromatic number at most $c$. By \cref{uchi-cchi}~(ii), there is a function $f$ such that every $G \in \GG$ can be $c$-coloured with clustering $f(\chi(G))$. By the hypothesis for $\GG$, there is some $k > 2c$, and some $G \in \GG$ with $\chi(G) = k$ and girth of $G$ greater than $f(k)$. As $G \in \GG$, one can $c$-colour $G$ with clustering at most $f(\chi(G)) = f(k)$. Since the girth of $G$ is greater than $f(k)$, every monochromatic component is acyclic and thus $2$-colourable. But then $G$ is $(2c)$-colourable, which is a contradiction.
\end{proof}

\begin{proof}[Proof of \cref{uchi-Hfree}]
    First suppose that $H$ is a forest. It is folklore that every graph with minimum degree at least $t$ contains every forest on $t + 1$ vertices. In particular, every graph in $\GG_H$ is $(\abs{V(H)} - 1)$-degenerate, and so $\GG_H$ has chromatic number (and underlying chromatic number) at most $\abs{V(H)}$. 
    Now assume that $H$ contains a cycle of length $\ell$. Let $\GG$ be the class of graphs with girth greater than $\ell$, so $\GG \subseteq \GG_{H}$. There are graphs of arbitrarily high chromatic number and girth~\citep{Erdos59}, so, by \cref{uchi-highgchi}, $\GG$ has unbounded underlying chromatic number.
\end{proof}

The graphs $C_{c, \ell}$ and $G_{c, \ell}$ that we defined in \cref{LowerBounds} also provide useful lower bounds.

\begin{lem}\label{uchi-lower}
The classes $\{G_{c, \ell} \colon \ell \in \NN\}$ and $\{C_{c, \ell} \colon \ell \in \NN\}$ have underlying chromatic number $c + 1$.
\end{lem}

\begin{proof}
First note that for any graph $G$, $\ell G$ has chromatic number $\chi(G)$ and $\widehat{\ell G}$ has chromatic number $\chi(G) + 1$. A simple induction then gives $\chi(G_{c, \ell}) = \chi(C_{c, \ell}) =  c + 1$ and so $\{G_{c, \ell} \colon \ell \in \NN\}$ and $\{C_{c, \ell} \colon \ell \in \NN\}$ have bounded chromatic number. Hence, by \cref{uchi-cchi}~(iii), their underlying chromatic numbers are the same as their clustered chromatic numbers. But these clustered chromatic numbers are $c + 1$.
\end{proof}

\begin{lem}
\label{LineGraphs}
Line graphs have unbounded underlying chromatic number.
\end{lem}

\begin{proof}
The following argument is adapted from \citep{ADOV03,HST03}. Suppose for the sake of contradiction that there exists $c$, such that for every graph $G$, $L(G)$ is contained in $H\boxtimes K_{f( \chi(L(G) )}$, where $\chi(H)\leq c$. Since $\chi(L(G))\in\{\Delta(G),\Delta(G)+1\}$, this is equivalent to saying
$L(G)$ is contained in $H\boxtimes K_{f( \Delta(G))}$.
Say $G$ is $\Delta$-regular with $n$ vertices and girth $g$, where $\Delta= 2c$ and $g>f(\Delta)$. (\citet{ES63} proved that such graphs exist.)\ So $L(G)$ is $(2\Delta-2)$-regular with $\Delta n/2$ vertices. By assumption,
$L(G)$ is $c$-colourable with clustering $f(\Delta)$.
Some colour class $X$ has at least $\Delta n / 2c= n$ vertices of $L(G)$. In $G$, $X$ corresponds to a set of at least $n$ edges. Hence $X$ contains a cycle in $G$, which has length at least $g$.
Therefore $f(\Delta)\geq g$, which is a contradiction. So line graphs have unbounded underlying chromatic number.
\end{proof}

Recall that $\II_H$ denotes the class of graphs with no induced subgraph isomorphic to $H$.

\begin{thm}
$\II_H$ has bounded underlying chromatic number if and only if $H$ is a path on at most three vertices or $H$ is an independent set.
\end{thm}

\begin{proof}
First suppose that $\II_H$ has bounded underlying chromatic number.  
Note that $\GG_{H} \subseteq \II_{H}$, so $\GG_{H}$ has bounded underlying chromatic number, and so $H$ is a forest by \cref{uchi-Hfree}. By \cref{uchi-lower}, the class $\{C_{c, \ell} \colon c, \ell \in \NN\}$ has unbounded underlying chromatic number. By \cref{StandardExamples}~(vi), no graph in this class contains an induced $P_4$. Thus $H$ has no induced $P_4$. Hence each component of $H$ is a star. Since line graphs contain no induced $K_{1,3}$, \cref{LineGraphs} implies that $H$ contains no  $K_{1,3}$. So each component of $H$ is a path on at most three vertices. 
Suppose $H$ contains the disjoint union of $K_1$ and $K_2$. Then complete multipartite graphs (that is, blow-ups of complete graphs) contain no induced $H$, and $\II_H$ has unbounded underlying chromatic number by \cref{uchi-blowups}. Thus $H$ does not contain the disjoint union of $K_1$ and $K_2$ as an induced subgraph. It follows that $H$ is a path on at most three vertices or a graph with no edges. 

For the converse direction, first suppose $H$ is a path with at most three vertices. Then every component of every graph in $\II_H$ is a complete graph, and $\II_H$ has underlying chromatic number $1$. Now suppose $H$ has $k$ vertices and no edges. Then $\alpha(G)\leq k-1$ for every $G\in\II_H$, implying $G$ is contained in  $K_{(k-1)\chi(G)}$. Hence $\II_H$ has underlying chromatic number $1$. 
\end{proof}

}
\end{document}